\documentclass[12pt]{article}

\usepackage{fullpage,amsfonts,amsmath,amsthm,amssymb,mathrsfs,graphicx,epstopdf,soul}
\usepackage[top=2.54cm, bottom=2.54cm, left=2.54 cm, right=2.54 cm]{geometry}
 \newcommand{\lab}[1]{\label{#1}}                

 \usepackage[usenames,dvipsnames]{color}
\usepackage[shortlabels]{enumitem}

\newcommand{\remove}[1]{}
\newcommand\eqn[1]{(\ref{#1})}

\newcommand{\be}{\begin{equation}}
\newcommand{\bel}[1]{\begin{equation}\lab{#1}\ }
\newcommand{\ee}{\end{equation}}
\newcommand{\bea}{\begin{eqnarray}}
\newcommand{\eea}{\end{eqnarray}}
\newcommand{\bean}{\begin{eqnarray*}}
\newcommand{\eean}{\end{eqnarray*}}

\newtheorem{thm}{Theorem}

\newtheorem{con}[thm]{Conjecture}
\newtheorem{lemma}[thm]{Lemma}
\newtheorem{definition}[thm]{Definition}
\newtheorem{claim}[thm]{Claim}
\newtheorem{prop}[thm]{Proposition}
\newtheorem{obs}[thm]{Observation}

\newtheorem{remark}[thm]{Remark}

\def\proof{\noindent{\bf Proof.~}~}
\def\qed{~~\vrule height8pt width4pt depth0pt}
\def\ss{{\smallskip}}

\newcommand{\floor}[1]{\left\lfloor#1\right\rfloor}

\newcommand{\rk}[1]{\text{rank}(#1)}

\newcommand{\ind}[1]{1_{\{#1\}}}
\def\bin{{\bf Bin}}
\newcommand{\supp}[1]{{\text{supp}\left(#1\right)}}

\newcommand\core[1]{{#1_{\text{2-core}}}}
\newcommand\red[1]{{#1_{\text{red}}}}

\newcommand{\cl}[1]{{\text cl}(#1)}
\newcommand{\spn}[1]{\langle#1\rangle}
\newcommand{\cn}{/}
\newcommand{\dl}{\!\setminus\!}

\def\E{{\mathcal E}}

\def\G{{\mathcal G}}

\def\I{{\mathcal I}}

\def\P{{\mathcal P}}

\def\T{{\mathcal T}}

\def\Z{{\mathcal Z}}
\def\PG{{\text PG}}
\def\bF{{\mathbb F}}

\def\ex{{\mathbb E}}
\def\pr{{\mathbb P}}

\def\bfv{{\boldsymbol v}}

\def\bF{\mathbb{F}}

\def\Row{{\texttt{Row}}}

\def\eps{\epsilon}

\date{}

\title{Minors of matroids represented by sparse random matrices over finite fields}
\author{Pu Gao \thanks{Research supported by NSERC RGPIN-04173-2019} \\
University of Waterloo\\
pu.gao@uwaterloo.ca 
\and 
Peter Nelson \thanks{Research supported by NSERC RGPIN-03735-2023} \\
University of Waterloo\\
apnelson@uwaterloo.ca
}
\begin{document}
\maketitle

\begin{abstract}

Consider a random $n\times m$ matrix $A$ over the finite field of order $q$ where every column has precisely $k$ nonzero elements, and let $M[A]$ be the matroid represented by $A$. In the case that q=2, Cooper, Frieze and Pegden (RS\&A 2019) proved that given a fixed binary matroid $N$, if $k\ge k_N$ and $m/n\ge d_N$  where $k_N$ and $d_N$ are sufficiently large constants depending on N, then a.a.s. $M[A]$ contains $N$ as a minor. We improve their result by determining the sharp threshold (of $m/n$) for the appearance of a fixed matroid $N$ as a minor of $M[A]$, for every $k\ge 3$, and every finite field. 

\end{abstract}


\section{Introduction}

Random graphs were first introduced by Erd\H{o}s and R\'{e}nyi~\cite{erdds1959random,erdHos1960evolution} in 1959, marking them one of the most important subjects in the study of modern graph theory. In their seminal work~\cite{erdHos1960evolution},  they defined a random graph process to explore the evolution of random graphs. The process starts with an empty graph $G_0$ on $n$ vertices. At every step $1\le m\le \binom{n}{2}$, an edge is added to $G_{m-1}$, which is chosen uniformly from all edges that are not present in $G_{m-1}$. Researchers have focused on investigating the earliest occurrence of specific structures in this random graph process $(G_m)_{0 \leq m \leq \binom{n}{2}}$, such as trees of a certain order, cycles of particular lengths, Hamilton cycles, connected graphs, $k$-connected subgraphs, $k$-cores, and more. Since its introduction, this random graph process has undergone extensive examination and has greatly influenced research over the past half-century and continues to do so.

We may view $(G_m)_{0\le m\le \binom{n}{2}}$ as a random process for graphic matroids, where the set of the edges in $G_m$, and the family of forests in $G_m$ is the ground set, and the family of independent sets, of the corresponding matroid. In other words, the random graph process induces a random matroid process $(M[I(G_m)])_{0\le m\le \binom{n}{2}}$, where $I(G_m)$ denotes the incidence matrix of $G_m$, and $M[A]$ denotes the matroid represented by matrix $A$.

The column vectors of $I(G_m)$ are binary vectors with exactly two entries equal to 1. A natural generalisation of the random graphic matroid process described as above, is to allow column vectors to have support size (i.e.\ the number of nonzero entries) different from two, and to permit the nonzero entries to take value from other fields. Motivated by this, we introduce the following more general random matroid process.

Let ${\mathbb F}$ be a field, ${\mathbb F}^*={\mathbb F}\setminus \{0\}$ and $k\ge 3$ be a fixed integer. The case $k=2$ corresponds to minors of random graphs, which is already well understood; see Remarks~\ref{remark:k} and~\ref{remark:k=2} below. Let $\P: ({\mathbb F}^*)^k\to [0,1]$ be a permutation-invariant probability distribution over $({\mathbb F}^*)^k$; i.e.\ ${\mathcal P}(x_1,\ldots,x_k)={\mathcal P}(x_{\pi(1)},\ldots, x_{\pi(k)})$ for any permutation $\pi$ in the symmetric group $S_k$. Let $\bfv\in {\mathbb F}^n$ be the random vector defined as follows. First choose  a uniformly random subset $K\subseteq [n]$ with $|K|=k$; then set the value of $(v_j)_{j\in K}$ by the law $\P$, whereas $v_j$ is set to 0 for all $j\notin K$. Let $(\bfv_m)_{m\ge 1}$ be a sequence of random vectors each of which is an independent copy of $\bfv$. Let $(A_m)_{m\ge 0}$ be a sequence of random matrices where $A_0$ is empty and for each $m\ge 1$, $A_m=[\bfv_1,\ldots,\bfv_m]$ is obtained by including the first $m$ column vectors $\bfv_1,\ldots,\bfv_m$.  Thus, $A_m$ is a random matrix over ${\mathbb F}$ where every column vector has support size equal to $k$. We study the evolution of the matroids $(M[A_m])_{m\ge 0}$ represented by the sequence of random matrices $(A_m)_{m\ge 0}$.


A special case of this random matroid process where ${\mathbb F}={\mathbb F}_2$ (the binary field) has been introduced and studied earlier by Cooper, Frieze and Pegden~\cite{cooper2019minors}. Note that with ${\mathbb F}={\mathbb F}_2$, the only possible $\P$ is $\P(1,\ldots, 1)=1$. In their work, they studied the appearance of a fixed binary matroid $N$ as a minor of $M[A_m]$. While the matroid growth rate theorem by Geelen, Kung, Kabell and Whittle~\cite{geelenkabell2009, geelen2009growth} (see also Theorem~\ref{thm:rate} below) immediately implies that  $N$ appears as a minor of $M[A_m]$ for some $m=O(n^2)$, Cooper, Frieze and Pegden proved that $m=O(n)$ suffices for the appearance of an $N$-minor for this random $\bF_2$-representable matroid. More precisely, given a binary matroid $N$, there exist sufficiently large constants $k_N$ and $d_N$ such that if $k\ge k_N$ and $m\ge d_N n$, then asymptotically almost surely (a.a.s.) $M[A_m]$ contains $N$ as a minor. As further questions, they asked ``to reduce $k$, perhaps to 3, and to get precise estimates for the number of columns needed for some fixed matroid, the Fano plane for example.'' They also asked minors of $M[A_m]$ over fields other than ${\mathbb F}_2$. In this paper we answer these questions. We characterise the phase transition for the appearance of an $N$-minor, for every $N$ of fixed order. A very surprising discovery from our result is that this phase transition solely depends  on $k$, and to a significant extent, is independent of the minor $N$, the field ${\mathbb F}$, or the distribution $\P$. Define
\begin{eqnarray}
\rho_{k,d}&=&\sup\{x\in [0,1]: x= 1-\exp(-d x^{k-1})\} \lab{eq:rho}\\
d_k&=&\inf\{d>0: \rho_{k,d}-d\rho_{k,d}^{k-1}+(1-1/k)d\rho_{k,d}^k<0\} \lab{eq:dk}
\end{eqnarray}
\begin{remark}\label{remark:k}
    For every $k\ge 3$, it can be shown that $d_k$ is well defined; see e.g.\ the remark after~\cite[Theorem 1.1]{ayre2020satisfiability}. For $k=2$, it is straightforward to check that $\rho_{2,d}=0$ for all $d<1$, $\rho_{2,d}>0$ for all $d>1$, and $\rho_{2,d}-d\rho_{2,d}+\frac{1}{2} d\rho_{2,d}^2<0$ for all $d>1$, and thus $d_2=1$. For readers who are familiar with random $k$-uniform hypergraphs $\G_k(n,m)$ with $n$ vertices and $m$ hyperedges, for $k\ge 3$, $d_k$ is precisely the critical density of $km/n$ where the 2-core of $\G_k(n,m)$ has more hyperedges than vertices, and $d_2=1$ is precisely the critical density of $2m/n$ where $\G_2(n,m)$ has a giant component. It has been well understood that all graph minors of fixed order appear simultaneously in $\G_2(n,m)$ when $m/kn$ exceeds $d_2$; see e.g.\  a proof in~\cite[Section 4]{luczak1994structure}, based on an argument initially proposed by Svante Janson.   
\end{remark}

Our first result characterises the phase transition of the appearance of minors for $A_m$ over prime fields.

\begin{thm}\lab{thm:main} Let $k\ge 3$. Suppose that  ${\mathbb F}={\mathbb F}_p$ where $p$ a prime number, and let ${\mathcal P}$ be any permutation-invariant distribution on $({\mathbb F}_p^*)^k$. Let $N$ be a fixed simple ${\mathbb F}_p$-representable matroid. 
Then, for every fixed $\eps>0$ the following hold:
\begin{enumerate}[(a)]
\item[(a)] If $m<(d_k/k-\eps)n$ then a.a.s.\ $M[A_m]$ does not contain $N$ as a minor if $N$ has a circuit; 
\item[(b)] If $m>(d_k/k+\eps)n$ then a.a.s.\ $M[A_m]$ contains $N$ as a minor.  
\end{enumerate}
\end{thm}

\begin{remark}
Without loss of generality we may always assume that $N$ has a circuit, as otherwise, all elements in $N$ are independent, and consequently, a.a.s.\ $N$ appears as a minor of $M[A_m]$ precisely at the step $m=\rk{N}=|N|$, as a.a.s.\ $\bfv_1,\ldots, \bfv_h$ are linearly independent if $h=O(1)$.
\end{remark}

Theorem~\ref{thm:main} is false for non-prime finite fields.  As a counterexample, consider ${\mathcal P}$ where ${\mathcal P}(1,\ldots,1)=1$ and $q=4$. Let $N$ be any ${\mathbb F}_4$-representable matroid which is not ${\mathbb F}_2$-representable. Then  for every $m$, $M[A_{m}]$ is a binary matroid, which cannot contain $N$ as a minor. Our next result confirms that Theorem~\ref{thm:main} is true for non-prime finite fields after imposing some weak condition on ${\mathcal P}$ .

\begin{thm}\lab{thm:main2} Let $k\ge 3$.  Suppose that  ${\mathbb F}={\mathbb F}_q$ where $q$ is a prime power, and ${\mathcal P}$ is a permutation-invariant distribution on $({\mathbb F}_q^*)^k$ such that ${\mathcal P}(\sigma_1,\ldots,\sigma_k)>0$ for every $(\sigma_1,\ldots,\sigma_k)\in ({\mathbb F}_q^*)^k$. Let $N$ be a fixed simple ${\mathbb F}_q$-representable matroid that contains a circuit.   
Then, for every fixed $\eps>0$ the following hold:
\begin{enumerate}[(a)]
\item[(a)] If $m<(d_k/k-\eps)n$ then a.a.s.\ $M[A_m]$ does not contain $N$ as a minor; 
\item[(b)] If $m>(d_k/k+\eps)n$ then a.a.s.\ $M[A_m]$ contains $N$ as a minor.  
\end{enumerate}
\end{thm}

\begin{remark}\label{remark:k=2}
The stipulation that $k = 3$ in the above results is necessary. When $k = 2$, the matroid $M[A_m]$ is defined to have a $\mathbb{F}_q$-representation in which each column has support at most $2$; matroids with this property are known as ${\mathbb F}_q$-frame matroids, and they form a minor-closed class. It follows that, if $k =2$ and $N$ is \emph{not} an ${\mathbb{F}}_q$-frame matroid, then $M[A_m]$ will never have an $N$-minor. Conversely, we believe that if $k=2$ and $N$ \emph{is} a $\mathbb{F}_q$-frame matroid of corank at least $2$, then $N$ will satisfy the conclusion of Theorem~\ref{thm:main2} for graph-theoretic reasons similar to those outlined in Remark~\ref{remark:k}. The proof is completely different from that of $k\ge 3$ in this paper, and is not trivial; we will discuss that in future work.
\end{remark}

A significant part of the proof for Theorem~\ref{thm:main2} is the same as for Theorem~\ref{thm:main}, and indeed holds for any general field ${\mathbb F}$. Thus we think that results like Theorem~\ref{thm:main2} would be true for general fields by restricting to certain distributions $\P$. See more discussions in Remark~\ref{r:field} below. As a concrete example, we make the following conjecture for the minors of $M[A_m]$
over ${\mathbb Q}$.

\begin{con}\lab{con:rational} Let $k\ge 3$.  Suppose that  ${\mathbb F}={\mathbb Q}$, the field of rational numbers, and ${\mathcal P}(1,\ldots,1)=1$. Let $N$ be a fixed simple ${\mathbb Q}$-representable matroid that contains a circuit.  Then $d_k/k$ is the threshold (of $m/n$) for the appearance of $N$ as a minor of $M[A_m]$.
\end{con}

If we were to change $\mathbb{Q}$ to $\mathbb{R}$ everywhere, this conjecture would fail if $N$ were not chosen to be $\mathbb{Q}$-representable, since any $M[A_m]$ generated by the process would be $\mathbb{Q}$-representable, as would all its minors. However, we believe that arbitrary minors should appear as long as the random entries we are using allow for it. We make a stronger conjecture to that effect. 

\begin{con}\lab{con:fieldgen} Let $k\ge 3$.  Let $\mathbb F$ be a field, and let $\mathcal{P}$ be a permutation-invariant probability distribution on $\bF^k$. For each $r \in \mathbb{F}$, let $\beta_r$ denote the probability that a vector chosen according to $\mathcal{P}$ contains two nonzero entries whose ratio is $r$. Let $\bF_0$ be the subfield of $\mathbb{F}$ generated by the set $G = \{r : \beta_r > 0\}$. Let $N$ be a fixed simple $\mathbb{F}_0$-representable matroid that contains a circuit. Then $d_k/k$ is the threshold for the appearance of $N$ as a minor of $M[A_m]$.
\end{con}

\begin{remark}\lab{r:fieldgen}
    \begin{enumerate}
        \item[(a)] 
Note that we always have $1 \in G$, so $\mathbb{F}_0$ contains $\mathbb{Q}$ in the characteristic-zero case and contains the prime subfield of $\bF$ in the positive characteristic case. Theorem~\ref{thm:main2} treats a special case where $\bF$ is finite, and every vector in $(\bF^*)^k$ appears with positive probability in $\P$, which gives the hypotheses of Conjecture~\ref{con:fieldgen} with $G = \bF^*$ and hence $\bF_0 = \bF$. So our result confirms a special case of the above. In fact, our techniques can be extended to verify the above conjecture for all finite $\bF$. 
\item[(b)] Suppose $m<(d_k/k-\eps)n$. If the column vectors of $A_m$ are not a.a.s.\ linearly independent then the subcritical case of  Conjecture~\ref{con:fieldgen} cannot be true. However, there is no result in the literature yet guaranteeing this for an arbitrary field $\bF$. The best known result regarding the rank of $A_m$ is given below in Theorem~\ref{thm:rank}, which does not imply linear independence of column vectors of $A_m$. In the case that ${\mathbb F}$ is a finite field, the linear independence of column vectors of $A_m$ has been shown~\cite[Theorem 1.1]{ayre2020satisfiability}. If $\bF={\mathbb Q}$ and $\P$ is supported on $(1,\ldots,1)$ only, then the linear independence of column vectors of $A_m$ in ${\mathbb Q}$ follows as a corollary of~\cite[Theorem 1.1]{ayre2020satisfiability}, since entries in $A_m$ are integers and linear dependency in ${\mathbb Q}$ would imply the linear independence in $\bF_p$ for any prime number $p$. The same argument holds if $\P$ is supported on $S^k$ for certain properly chosen subsets $S\subseteq {\mathbb Q}$. It would already be very interesting to prove or refute that the column vectors of $A_m$ are linearly independent for every permutation-invariant probability distribution on ${\mathbb Q}^k$. 
    \end{enumerate}
\end{remark}
In addition to minors, there are several other fundamental questions that warrant answers when considering random matroids. These include determining the rank, the number of bases, the occurrence of circuits of specific lengths, the connectivity, the number of submatroids isomorphic to a give matroid, and minors of matroids with an order that is a function of the growing parameter $n$, the automorphism group of the random matroid, etc. The rank of $M[A_m]$ can be deduced by examining the rank of the corresponding $A_m$ matrix, as established through previous research on the rank of random sparse matrices~\cite{coja2020rank}.

\begin{thm}(Corollary of~\cite[Theorem 1.1]{coja2020rank})
For any field ${\mathbb F}$, integer $k\ge 1$, permutation-invariant distribution $\P$, nonnegative real $d$, and $m\sim d n/k$,
\[
\lim_{n\to\infty}\frac{\rk{M[A_m]}}{n} = 1- \max_{\alpha\in [0,1]} \left\{\exp(-d \alpha^{k-1}) -\frac{d}{k}(1-k\alpha^{k-1}+(k-1)\alpha^k)\right\}.
\]
\label{thm:rank}
\end{thm}
With some calculations, it is easy to see that the maximisation problem on the right hand side above is attained at $\alpha=0$ for $d<d_k$. That is, if $d<d_k$ then a.a.s.\ $\rk{M[A_m]}\sim m$. As remarked above, in the case that ${\mathbb F}$ is a finite field, a stronger full-column-rank result~\cite[Theorem 1.1]{ayre2020satisfiability} is known if $d<d_k$, and this is what we use in Section~\ref{sec:sub} to prove Theorems~\ref{thm:main}(a) and~\ref{thm:main2}(a).

While the other questions such as the connectivity, the circuits, etc.\ have not been specifically studied for $M[A_m]$, they have been investigated in other random matroid models, which we will discuss in detail in the following subsection.

\subsection{Other random matroid models}
There are two other random matroid models that have been studied in the literature. 
Mayhew, Newman, Welsh and Whittle~\cite{mayhew2011asymptotic} introduced the uniform model for random matroid on $n$ elements, and they proposed various conjectures concerning the connectivity, symmetry, minors, rank and representability. In their work~\cite{mayhew2011asymptotic}, they proved that the probability of a uniformly random matroid being connected is at least 1/2.
In a related study~\cite{lowrance2013properties},  Lowrance, Oxley, Semple and Welsh proved that almost all matroids are simple, cosimple, and 3-connected. They obtained bounds on the rank,  and estimated the number of bases, the number of circuits with given length, and the maximum circuit size.
The study of the uniform model of random matroids eventually boils down to the enumeration of matroids on $n$ elements. 
Knuth~\cite{knuth1974asymptotic} established a lower bound for the number of matroids with $n$ element, while the current best upper bound is given by Bansal, Pendavingh, and van der Pol~\cite{bansal2015number}.
By counting representably matroids, Nelson~\cite{nelson2016almost} proved that almost all matroids are non-representable.
Pendavingh and Van der Pol~\cite{pendavingh2018number} proved that if a uniform random matroid  $M$ has rank $r$, then almost all $r$-subsets of $E(M)$ are beses of $M$. They derived several results on girth, connectivity, and other related properties from this observation.

On the other hand,
Kelly and Oxley~\cite{kelly1982asymptotic,kelly1982threshold} introduced the concept of random representable matroid over finite fields and specifically the model $\PG(n-1,q,p)$. In this model, each vector in the projective space $\PG(n-1,q)$ is independently included into the ground set of the random matroid with a probability $p$, analogous to the binomial random graph model $\G(n,p)$. They focused their studies on various aspects of these random matroids, including the rank of the largest projective geometry occurring as a submatroid, as well as other submatroids and independent sets. Subsequent work by Kordecki~\cite{kordecki1988strictly,kordecki1996small}, Kordecki and {\L}uczak~\cite{kordecki1991random,kordecki1999connectivity} delved into further exploration of the connectivity, circuits and small submatroids. In~\cite{kelly1984random} Kelly and Oxley introduced a slightly different model $M[U_{n,m}]$,  where $U_{n,m}$ is a uniformly random $n\times m$ matrix over ${\mathbb F}_q$. This model is asymptotically equivalent to $\PG(n-1,q,p)$ unless $m$ is close to $q^n$ (or equivalently when $p$ is close to 1 in $\PG(n-1,q,p)$.  In this slightly different model,  Kelly and Oxley studied the rank, connectivity and circuit sizes of $M[U_{n,m}]$. More recently,  
Altschuler and Yang~\cite{altschuler2017inclusion} determined the threshold for the appearance of fixed minors  in $M[U_{n,m}]$.

\section{Preliminaries}

We refer the readers to Oxley~\cite{oxley2006matroid} for basic definitions of matroids. In particular, if $M = (E, \I)$ is a matroid and $X \subseteq E$, then the contraction and deletion of $X$ in $M$ are defined respectively by $M \dl X = (E - X, \{I \in \I : I \cap X = \varnothing\})$ and 
\begin{align*}
M \cn X &= (E - X, \I^*), \quad\mbox{where}\\
\I^* &= \{I \in \I : I \cap X = \varnothing, I \cup J \in \I\ \mbox{for some maximal independent subset $J$ of $X$}\}.
\end{align*}
The independent sets of $M / X$ turn out not to depend on the choice of $J$. A \emph{submatroid} of $M$ (also called a ``restriction'') is any matroid of the form $M \dl X$, and a \emph{minor} of $M$ is a matroid of the form $M \cn X \dl Y$. Contraction and deletion commute in the sense that $M \cn X_1 \dl Y_1 \cn X_1 \dl Y_2 = M \cn (X_1 \cup X_2) \dl (Y_1 \cup Y_2)$, and so any sequence of deletions and contractions gives a minor. 
 
Let ${\mathbb F}$ be a field.
A matroid $M=(E,\I)$ is \emph{${\mathbb F}$-representable} if there is a matrix $A=[a_p]_{p\in E}$ over ${\mathbb F}$ whose columns are indexed by elements in $M$ such that $S\subseteq E$ is in $\I$ if and only if $\{a_p: p\in S\}$ are linearly independent. We call such a matrix $A$ an \emph{$\mathbb{F}$-representation} of $M$.  A matroid $M$ is called \emph{simple} if for every $p_1,p_2\in E$, $\{p_1,p_2\}\in \I$. In other words, if $A=[a_p]_{p\in E}$ is an $\bF$-representation of a simple matroid $M$ over ${\mathbb F}$, then the columns of $A$ are nonzero and pairwise non-parallel. 

Given a matrix $A$ over ${\mathbb F}$, let $M[A]$ denote the matroid represented by $A$. We say two matrices $A$ and $B$ are row equivalent, denoted by $A\sim B$, if $B$ is obtained from $A$ by performing a sequence of elementary row operations. Note that an elementary row operation does not change column dependencies, and thus $M[A]=M[B]$ if $A\sim B$.

We wish to consider deletions and contractions of representable matroids. Given a matrix $A$ over ${\mathbb F}$ with column set $E$ and a subset $X \subseteq E$, let $A_X$ denote the submatrix of $A$ consisting of the columns in $X$, and let $A_{\setminus X} = A_{E - X}$. Let $A_{/X}$ be  any matrix obtained by first finding $B\sim A$ such that
\[
B_X=\left[\begin{array}{cc}I_{a\times a} & * \\ 0 & 0 \end{array}\right], \quad \text{ where $a=\rk{A_X}$}
\]
and then deleting from $B$ the $a$ rows where the identity matrix $I_{a\times a}$ lies, together with all the columns in $X$. The following standard result shows that these matrices gives representations of deletions and contractions of $M$. 
\begin{prop}
    If $A$ is a matrix over $\mathbb{F}$ with column set $E$, and $X \subseteq E$, then $M[A] \dl X = M[A_{\setminus X}]$, and $M[A] / X = M[A_{/X}]$, regardless of the choice of $A_{/X}$.
\end{prop}




Let $q$ be a prime power and let ${\mathbb F}_q$ denote the finite field of order $q$. A special class of representable matroids is the projective geometry $\PG(t-1,q)$, which is the matroid represented by the matrix $A$ whose column vectors correspond to the set of 1-dimensional subspaces of ${\mathbb F}_q^t$ (i.e.\ $A$ contains exactly one nonzero vector in each 1-dimensional subspace). The concept of $\PG(t-1,q)$ bears a resemblance to the notion of a simple complete graph, and as a result, they play a central role in the examination of submatroids and minors of matroids, in contrast to the study of subgraphs and minors of graphs. By the following simple observation, in order to find any specific ${\mathbb F}_q$-representable rank-$t$ matroid as a minor, it suffices to find $\PG(t-1,q)$ as a minor.

\begin{obs} Let $q$ be a prime power. Then,
$\PG(t-1,q)$ contains every simple rank-t ${\mathbb F}_q$-representable matroid as a submatroid.
\end{obs}


In this paper, all asymptotics refer to $n\to\infty$. For two sequences of real numbers $(a_n)$ and $(b_n)$, we write $a_n = O(b_n)$ if there is $C > 0$ such that $|a_n| < C|b_n|$ for every $n$. We say $a_n =\Omega(b_n)$ if $a_n >0$ and $b_n =O(a_n)$. We say $b_n =o(a_n)$ if $a_n >0$ and $lim_{n\to\infty} b_n/a_n = 0$. We say $b_n=\omega(a_n)$ if $b_n>0$ and $a_n=o(b_n)$. We say a sequence of events $A_n$ holds a.a.s. if $\pr(A_n) = 1 - o(1)$. 

\section{Proof of Theorem~\ref{thm:main}(a) and Theorem~\ref{thm:main2}(a)}
\label{sec:sub}

By~\cite[Theorem 1.1]{ayre2020satisfiability} (note that our matrix $A$ in this paper corresponds to the transpose of the matrix in~\cite{ayre2020satisfiability} ), a.a.s.\ all columns of $A$ are linearly independent under the hypotheses of Theorem~\ref{thm:main}(a) or Theorem~\ref{thm:main2}(a). Thus $M[A]$ is a free matroid, which cannot contain a minor that has a circuit. \qed

\section{Proof of Theorem~\ref{thm:main}(b) and Theorem~\ref{thm:main2}(b)}
\label{sec:sup}

Most proofs in this section are independent of ${\mathbb F}$ and $\P$. 
Without specification, ${\mathbb F}$ is an arbitrary field, and $\P$ is an arbitrary permutation-invariant distribution on $(\mathbb F^*)^k$. Let $t$ be a fixed integer and let $N$ be a simple ${\mathbb F}$-representable rank-$t$ matrix. We aim to find $N$ as a minor of $M[A_m]$ where $m=(d_k/k+\eps)n$ for some $\eps=o(1)$. Recall that $\bfv\in {\mathbb F}^n$ is defined as a random vector by first choosing  a uniformly random subset $K\subseteq [n]$ with $|K|=k$ and then setting the value of $(v_j)_{j\in K}$ by the law $\P$, whereas $v_j$ is set to 0 for all $j\notin K$.

In this section, we use the notation $A^1, A^2, \ldots$ to represent matrices indexed by superscripts.  It is important to note that the proofs in this paper do not involve powers of any matrix. This clarification is provided to ensure that readers do not confuse the notation with matrix powers.

We construct $A_m$ by starting with $A^1\sim A_{m_1}$ where $m_1$ is slightly smaller than $d_k n$. Then, we carry out multiple rounds of sprinkling random vectors to build the matrix $A_m$. Throughout the process we leverage the existing structures present in $A^1$ as well as those emerging from the successive rounds of sprinkled vectors. These structures aid us in identifying the columns of $A_m$ to be  contracted, and ultimately enable us to find the desired minor.


{\bf Step 1} ({\em Start with a subcritical matrix}): Let $A^1\sim A_{m_1}$ where $m_1=(d_k/k-\eps_1)n$ where $\eps_1>0$ and $\eps_1=o(1)$. The following lemma follows by~\cite[Theorem 1.1]{ayre2020satisfiability}.

\begin{lemma}\lab{lem:initial} Let $\bF$ be a finite field.
Then, a.a.s.\ all columns of $A_1$ are linearly independent, provided that $\eps_1$ approaches to 0 sufficiently slowly.     
\end{lemma}

\begin{definition}\lab{def:2core}
Given a matrix $A$, the 2-core of $A$, denoted by $\core{A}$, is defined by the matrix obtained by repeatedly deleting the all-zero rows, and deleting the rows that contain exactly one nonzero entry, as well as the column where that nonzero entry lies in.
\end{definition}

Let $\core{(A^1)}$ be the 2-core of $A^1$ and let $n_1$ denote the number of columns in $\core{(A^1)}$. By definition, every column of $\core{(A^1)}$ has exactly $k$ nonzero entries, and every row of $\core{(A^1)}$ has at least two nonzero entries. Next, we prove that $\core{(A^1)}$ is close to a square matrix.
\begin{lemma}\lab{lem:rows}
A.a.s.\ $n_1=\Theta(n)$ and the number of rows in $\core{(A^1)}$ is $(1+\Theta(\eps_1))n_1$. 
\end{lemma}

By possibly rearranging the columns and rows of $A^1$, we may write 
\[
A^1=\left[\begin{array}{cc}\core{(A^1)}\ & *\\ 0 & * \end{array}\right],
\]
where $*$ denote entries in $A$ belonging to columns and rows that do not belong to $\core{(A^1)}$. Let $\Row(\core{(A^1)})$ denote  the set of rows in $\core{(A^1)}$, which is the set of rows of $A^1$ that are contained in $\core{(A^1)}$.

{\bf Step 2} ({\em First round of sprinkling random vectors}):  By Lemma~\ref{lem:rows} we may assume that all columns of $\core{(A^1)}$ are linearly independent. Let $\eta=o(1)$ be chosen so that $\eta=o(\eps_1)$. Let $A^2$ be the matrix obtained by adding $\eta n$ independent random column vectors to $A^1$, each distributed as $\bfv$. Given a subset $X$ of columns of a matrix $A$, let $\spn{A_X}$ and $\spn{A}$ denote the subspaces spanned by the column vectors in $A_X$ and in $A$ respectively.  Let $U$ be a subset of columns of $A^2$ constructed as follows. Given a vector $u\in {\mathbb F}^n$, let $\supp{u}$ denote the set of entries of $u$ whose value is nonzero.

\begin{enumerate}[(a)]
\item Let $U_0$ be the set of columns in $\core{(A^1)}$; denote these column vectors by $v_1,\ldots, v_{n_1}$.
\item Let $u_1,\ldots, u_{\eta n}$ be the $\eta n$ column vectors in $A^2\setminus A^1$ whose columns are indexed by $p_1,\ldots, p_{\eta n}$. For $i=1,\ldots, \eta n$, if $\supp{u_i}\subseteq \Row(\core{(A^1)})$ and $u_i\notin \spn{(A^2)_{U_{i-1}}}$, then $U_{i}=U_{i-1}\cup \{p_i\}$; otherwise $U_i=U_{i-1}$. 
\item Let $U=U_{\eta n}$.
 \end{enumerate}

Let $A^3$ be the submatrix of $A^2$ obtained by deleting all columns that are not in $U$, i.e.\ $A^3=(A^2)_{U}$. Note that by the construction of $U$, $A^3$ contains all column vectors in $\core{(A^1)}$. 

\begin{lemma}\lab{lem:U} 
Let $v\sim\bfv$ be independent of $A^3$. Then, 
\[
\pr(v\in \spn{A^3}\mid v\in \Row(\core{(A^1)}))= 1-O(\eps_1/\eta).
\]
\end{lemma}

By the construction of $U$, the columns of $A^3$ are all linearly independent, and they contain $U_0$, the set of all columns of $\core{(A^1)}$. We may, by permuting the rows of $A^3$ if necessary, write
\[
A^3=\left[
\begin{array}{c}
A'' \\
* 
\end{array}
\right]
\]
where 
$A''$ is an invertible $|U|\times |U|$ matrix. 
By Lemma~\ref{lem:rows} we may assume that 
\[
\mbox{$A''$ contains all but at most $O(\eps_1)n$ rows of $\Row(\core{(A^1)})$.}
\]

{\bf Step 3} ({\em The linear operator turning $A''$ to $I$}): Let $B=(A'')^{-1}$. Let $b_1,\ldots,b_{n'}$ be the rows of $B$ where $n'=|U|$. Given $J\subseteq U$, let $B^J$ denote the matrix consisting of rows $b_j$, $j\in J$.
 We prove the following key lemma about $B$. 
\begin{lemma}\lab{lem:key}
Let $r>0$ be a fixed positive integer. There exists a fixed $\delta>0$ such that a.a.s.\ for every $J\subseteq [r]$ and for every nonzero $\alpha\in {\mathbb F}^{J}$, $\sum_{j\in J}\alpha_jb_j$ has at least $\delta n'$ nonzero elements.
\end{lemma}

{\bf Step 4} ({\em The second round of sprinkling random vectors}): Let $A^4$ be the matrix obtained by adding $\eps_1 n$ random column vectors $v_1,\ldots, v_{\eps_1 n}$ to $A^2$, each as an independent copy of $\bfv$. Let $(A')^4$ be the corresponding matrix by adding $v_1,\ldots, v_{\eps_1 n}$ to $A^3$ instead. Thus $(A')^4$ is a submatrix of $A^4$. 

 By Lemma~\ref{lem:rows}, $\pr(\supp{v_i}\subseteq \Row(\core{(A^1)}))=\Theta(1)$ for every $1\le i\le \eps_1 n$. By Lemma~\ref{lem:U}, $\pr(v_i \in \spn{(A^5)_U}\mid \supp{v_i}\subseteq \Row(\core{(A^1)}))=1-O(\eps_1/\eta)=1-o(1)$ for every $1\le i\le \eps_1 n$.  It follows by the standard Chernoff bound that a.a.s.\ $|V| =\Theta(\eps_1 n)$,
 where
 \[
 V=\{v\in \{v_1,\ldots, v_{\eps_1 n}\}: \supp{v}\subseteq\Row(\core{(A^1)}), v\in \spn{(A^5)_U}\}.
 \]
 
 Let $v'_1,\ldots,v'_{\eps_2 n}$ be the vectors in $V$. To match the dimension, write $v'_i=\left[\substack{x'_i\\y'_i}\right]$.
 Let $A^5$ be the submatrix of $(A')^4$ such that
\[
A^5=\left[
\begin{array}{ccc}
A'' & | & x'_1,\ldots, x'_{\eps_2 n}\\
A^* & | & y'_1,\ldots, y'_{\eps_2 n}
\end{array}
\right].
\]
That is, $A^5$ is obtained from $(A')^4$ by deleting all vectors in $\{v_1,\ldots,v_{\eps_1n}\}\setminus V$.
Thus, with appropriate row operations, i.e.\ multiplying $A^5$ by 
\[
\left[\begin{array}{cc}
B & 0 \\
-A^*B & I
\end{array}
\right]
\]
 to the left,
$A^5$ is row equivalent to 
\[
\left[
\begin{array}{ccccc}
I & | & B x'_1&\ldots& B x'_{\eps_2 n}\\
0 & | & y'_1-A^*Bx'_1&\ldots& y'_{\eps_2n}-A^*Bx'_{\eps_2n} 
\end{array}
\right].
\]
Note that for each $1\le i\le \eps_2 n$, $v'_i\in \spn{(A^5)_U}$ if and only if $y'_i-A^*Bx'_i=0$. Thus we immediately have the following observation.

\begin{obs}\lab{obs:Uspan}
For every $1\le j\le \eps_2 n$, $y_j'-A^*Bx'_j=0$.
\end{obs}

{\bf Step 5} ({\em Contraction}): Recall that $U$ is the set of columns of $A^3$, i.e.\ the set of columns in the submatrix $\left[ \substack{A''\\ A^*} \right]$ of $A^5$. Let $r=r(t)$ be a sufficiently large constant that depends only on $t$, the rank of $N$, and let $[r]$ denote the first $r$ columns of $U$. Let $A^6$ be the matrix obtained from $A^5$ by contracting all columns in $U\setminus [r]$. By Observation~\ref{obs:Uspan},
\begin{equation}
\rk{A^6}=r.\lab{A6rank}
\end{equation}
Without loss of generality, we may keep only the first $r$ rows of $A^6$ as the remaining rows are all-zero rows by Observation~\ref{obs:Uspan}, and deleting these rows does not change the matroid the matrix represents.
 The final step is to find the desired minor in $M'=M[A^6]$.

\begin{lemma}\lab{lem:minor} Let $\bF$ be a finite field. Then,
a.a.s.\ $M'$ contains any fixed simple ${\mathbb F}$-representable matroid $N$ as a minor. 
\end{lemma}

\subsection{Completing the proofs of Theorems~\ref{thm:main}(b) and~\ref{thm:main2}(b)}
First of all, notice that $A^4\sim A_{m}$, where $m=m_1+\eta n+ \eps_1 n=(d_k/k+\eta) n$ where $\eta=o(1)$. As $(A')^4$ is a submatrix of $A^4$ and $M'$ is a minor $(A')^4$, which contains $N$ as a minor by Lemma~\ref{lem:minor}, Theorems~\ref{thm:main}(b) and~\ref{thm:main2}(b) follow. \qed\ss

\begin{remark}\lab{r:field}
Step 1 holds for $\bF={\mathbb Q}$ with appropriate setting of $\P$ such as in Conjecture~\ref{con:rational}. See Remark~\ref{r:fieldgen}(b) for more discussions. The subsequent steps and Lemmas~\ref{lem:rows}, ~\ref{lem:U},~\ref{lem:key} and Observation~\ref{obs:Uspan} hold for any arbitrary field. The only part of the proof that specifically requires $\bF$ to be a finite field is for Lemma~\ref{lem:minor}. 
\end{remark}

We complete this section by proving Lemmas~\ref{lem:rows} and~\ref{lem:U}. We prove Lemma~\ref{lem:minor} assuming Lemma~\ref{lem:key} in Section~\ref{sec:minor}, where two distinct cases are discussed separately, depending on whether ${\mathbb F}$ is a prime field, or a field of order that is a prime power. Finally, Lemma~\ref{lem:key} is proved in Section~\ref{sec:key}.

\subsection{Proof of Lemmas~\ref{lem:rows} and~\ref{lem:U}}
\remove{
Let 
\begin{align}
f_t(\mu) & = e^{-\mu}\sum_{i\ge t} \frac{\mu^i}{i!}\\
h(\mu) & = \frac{\mu}{f_{1}(\mu)^{k-1}}\\
c_k &= \inf_{\mu>0} \frac{h(\mu)}{k}.
\end{align}
For any $c\ge c_k$, let $\mu(c)$ be the larger solution to
\[
\frac{h(\mu)}{k}=c.
\]
Let
\[
\alpha(c)=f_2(\mu(c)),\quad \beta(c)=\frac{1}{k}\mu(c)f_{1}(\mu(c)).
\]
The following result is from~\cite[Lemma 7]{gao2018stripping}, which follows by~\cite[Theorem 1.7]{kim2007poisson}.
\begin{lemma}
Let $c=m/n$ and let $A\sim A_{n,m,k,q,{\mathcal P}}$. A.a.s.\ $\core{A}$ has asymptotically $\alpha(c)n$ rows and $\beta(c)n$ columns. 
\end{lemma}
}

{\em Proof of Lemma~\ref{lem:rows}. } Recall $\rho_{k,d}$ and $d_k$ from~\eqn{eq:rho} and~\eqn{eq:dk}. Let $d_k^*=\inf\{d>0:\rho_{k,d}>0\}$. By~\cite[Theorem 1.6]{ayre2020satisfiability}, by letting $d=km/n$, the number of rows and columns in $\core{A}$ is a.a.s.\ asymptotic to $(\rho_{k,d}-d\rho_{k,d}^{k-1}+d\rho_{k,d}^k)n$ and $(d\rho_{k,d}^k/k)n$ respectively. Moreover, the function
\[
\pi_k: d \to \frac{d}{k}\rho_{k,d}^k\left(\rho_{k,d}-d\rho_{k,d}^{k-1}+d\rho_{k,d}^k\right)^{-1}
\]
is strictly increasing on $[d_k^*,\infty)$. By the definition of $d_k$, the monotonicity of $\pi_k$, and the assumption that $m=(d_k/k-\eps_1)n$, the assertion in the lemma follows immediately. 
\qed \ss

{\em Proof of Lemma~\ref{lem:U}. } Let $I=\Row(\core{(A^1)})$.
Since $\supp{v}$ is uniform over $\binom{[n]}{k}$, it follows that $\supp{v}$ is uniform over $\binom{I}{k}$ conditional on $\supp{v}\subseteq I$. Let $K>0$ be a sufficiently large constant. Suppose on the contrary that 
\[
\pr(v\notin \spn{(A^2)_U}\mid v\in \supp{\core{(A^1)}}) > K \eps_1/\eta.
\]
Note that $\pr(u_i\notin \spn{(A^2)_{U_{i-1}}}\mid u_i\in \supp{\core{(A^1)}})$ is non-increasing for $i$. It follows immediately that
\[
\pr(u_i\notin \spn{(A^2)_{U_{i-1}}}\mid u_i\in \supp{\core{(A^1)}}) > K \eps_1/\eta,\quad \mbox{for every $1\le i\le \eta n$.}
\]
Let $X_i=\ind{U_i=U_{i-1}\cup \{p_i\}}$. Then, $\sum_{i=1}^{\eta n} X_i$ stochastically dominates $\bin(\eta n, K \eps_1/\eta)$. By the Chernoff bound, 
\[
\pr\left(\sum_{i=1}^{\eta n} X_i<(K/2) \eps_1 n\right) \le \pr\Big(\bin(\eta n, K \eps_1/\eta)<(K/2) \eps_1 n\Big)=\exp(-\Omega(\eps_1 n)).
\]
Thus, a.a.s.\
$\sum_{i=1}^{\eta n} X_i \ge  (K/2) \eps_1 n$. It follows then that a.a.s.\ $|U|-|U_0|\ge (K/2) \eps_1 n$. But then the submatrix of $(A^2)_U$ restricted to the rows in  $\Row(\core{(A^1)})$ contains more columns than rows by Lemma~\ref{lem:rows}, contradicting with the fact that all column vectors in $U$ are linearly independent.
\qed
\ss

\section{Proof of Lemma~\ref{lem:minor}}
\label{sec:minor}

\begin{definition}
For an $r\times n'$ matrix $C$, we say that $C$ contains $\{u_1,\ldots, u_r\}$ as a $\delta$-dense basis for ${\mathbb F}_q^r$, if $\spn{u_1,\ldots, u_r}={\mathbb F}_q^r$, and for each $1\le i\le r$, there are at least $\delta n'$ column vectors in $C$ equal to $u_i$.
\end{definition}

Recall that $B^{[r]}$ is the matrix obtained by taking the first $r$ rows of $B$; thus $B^{[r]}$ is an $r\times n'$ matrix. By Lemma~\ref{lem:rows}, $n'\ge n_1=\Theta(n)$. We prove that a.a.s.\ $B^{[r]}$ has a $\delta$-dense basis, for some constant $\delta>0$, if ${\mathbb F}$ is a finite field.

\begin{lemma}\lab{lem:denseBasis} Let $q$ be a prime power and let $r>0$ be a fixed positive integer.  Then there exists a constant $\delta>0$ such that  a.a.s.\ 
$B^{[r]}$ contains a $\delta$-dense basis for ${\mathbb F}_q^r$.
\end{lemma}

\proof We prove this lemma assuming Lemma~\ref{lem:key}. Let $\delta>0$ be the constant whose existence is guaranteed by Lemma~\ref{lem:key}. We prove that $B^{[r]}$ contains a $\delta/q^r$-dense basis for $\bF_q^r$. We say $u$ is a $(\delta/q^r)$-dense column vector of $B^{[r]}$ if there are at least $(\delta/q^r) n'$ column vectors of $B^{[r]}$ equal to $u$.  We prove that if $u_1,u_2,\ldots, u_h$ are linearly independent $(\delta/q^r)$-dense column vectors of $B^{[r]}$ and $h<r$, then there exists a $(\delta/q^r)$-dense column vector $u_{h+1}$ of $B^{[r]}$ such that $u_1,u_2,\ldots,u_{h+1}$ are linearly independent. Note that this immediately implies the assertion of the lemma.

Let $C$ be the set of column vectors of $B^{[r]}$ that are in $\spn{u_1,\ldots, u_h}$. 
We prove that 
\begin{equation}
|C|\le (1-\delta)n'. \lab{eq:C}
\end{equation}
Since $u_1,\ldots, u_h$  are linearly independent, $\rk{B^{[r]}_C}=h<r$. Hence there exists a nonzero $\alpha\in \bF_q^r$ such that $\alpha^T C =0$. In other words, $\alpha^T B_{[r]}$ contains at most $n'-|C|$ nonzero elements. By Lemma~\ref{lem:key}, $n'-|C|\ge \delta n'$, confirming~\eqn{eq:C}. 
Since $|\bF_q^r|= q^r$, there exists a column vector $u$  such that $u\notin C$ and $u$ appears at least $(\delta/q^r) n'$ times in $B^{[r]}$. Then, $u$ is $(\delta/q^r)$-dense, and $u_1,\ldots, u_h,u$ are linearly independent. Thus we may choose $u_{h+1}=u$.\qed
\ss

\subsection{Proof of Lemma~\ref{lem:minor} when ${\mathbb F}={\mathbb F}_p$ where $p$ is prime}

We use the following growth rate theorem by Geelen, Kabell, Kung and Whittle~\cite[Corollary 1.5]{geelen2009growth} to find $N$ as a minor of $M'$.
\begin{thm}\lab{thm:rate}
Let $p$ be a prime. For every integer $t\ge 2$ there exists $C_t>0$ such that for all positive integer $r$, if $M$ is an ${\mathbb F}_p$-representable rank-$r$ matroid with at least $C_t r^2$ elements then $M$ 	contains $\PG(t-1,p)$ as a minor.
\end{thm}

{\em Proof of Lemma~\ref{lem:minor} (when  ${\mathbb F}={\mathbb F}_p$). }  Let $C_t$ be the constant in Theorem~\ref{thm:rate}.
We prove that a.a.s.\ $M'$ has rank $r$ and at least $\binom{r}{k}$ elements. Since $k\ge 3$, this implies that $M'$ contains at least $C_t r^2$ elements by choosing sufficiently large $r$ and
consequently, $M'$ contains $\PG(t-1,p)$ as a minor by Theorem~\ref{thm:rate}.  

Recall that $M'=M[A^6]$. By~\eqn{A6rank}, $\rk{M'}=r$.
By Lemma~\ref{lem:denseBasis}, let $u_1,\ldots, u_r$ be column vectors of $B^{[r]}$ that form a $\delta$-dense basis for $\bF_p^r$.
Let $\alpha=(\alpha_1,\ldots,\alpha_k)$ be such that $\P(\alpha_1,\ldots,\alpha_k)>0$. Let $u_J=[u_j]_{j\in J}$ be the matrix obtained by including all column vectors in $\{u_j: \ j\in J\}$. Define
\[
S=\{u_J \alpha^T: \ J\subseteq [r], |J|=k.\}
\]
Then $|S|= \binom{r}{k}$. It suffices to prove that $A^6$ contains all vectors in $S$.

 Fix $v\in S$,
For $1\le j\le \eps_2 n$, let $X^v_j$ be the indicator variable that 
\[
B^{[r]}v'_j = v, \quad \mbox{and}\ v'_j\in \spn{(A_5)_U}.
\]
Thus, $v$ is a column vector of $A^6$ if $\sum_{1\le j\le \eps_2 n} X^v_j\ge 1$. Suppose $v=\sum_{j=1}^k \alpha_j u_{i_j}$ for some $\{u_{i_1},\ldots, u_{i_k}\}\subseteq \{u_1,\ldots, u_r\}$. Since $u_1,\ldots, u_r$ are all $\delta$-dense, and $\P(\alpha_1,\ldots,\alpha_k)>0$, $\pr(B^{[r]}v'_j = v)=\Omega(\delta)$, since $\supp{v_j'}$ is uniform in $\binom{\I}{k}$, where $\I=\Row(\core{(A^1)})$. Thus,
\[
\pr(X^v_j=1)=\Omega(\delta)-\pr(v'_j\notin \spn{(A_5)_U})=\Omega(\delta),
\]
as $\pr(v'_j\notin \spn{(A_5)_U})=o(1)$ by Lemma~\ref{lem:U} and the choice that $\eps_1=o(\eta)$.
By the Chernoff bound,
\[
\pr\left(\sum_{1\le j\le \eps_2 n} X^v_j=0\right)=\exp(-\Omega(\eps_2 n)),
\]
and thus by the union bound, a.a.s.\ $\sum_{1\le j\le \eps_2 n} X^v_j\ge 1$ for every $v\in S$. Hence, a.a.s.\ $A^6$, the matrix representing $M'$, contains all $\binom{r}{k}$ vectors in $S$. 
 \qed


\subsection{Proof of Lemma~\ref{lem:minor} when ${\mathbb F}={\mathbb F}_q$ where $q$ is a prime power}

In this section, let $\bF=\bF_q$ where $q$ is a prime power. Consequently, we assume that $\P$ is a permutation-invariant distribution on $({\mathbb F}_q^*)^k$ such that ${\mathcal P}(\sigma_1,\ldots,\sigma_k)>0$ for every $(\sigma_1,\ldots,\sigma_k)\in ({\mathbb F}_q^*)^k$, as required in the hypotheses of Theorem~\ref{thm:main2}.

 Given integer $\ell \ge 1$, a matrix $A \in \bF^{R \times E}$ is said \emph{$\ell$-complete} if $|R| \ge \ell$ and, for every nonzero vector $v \in \bF^R$ with support size at most $\ell$, there is a nonzero column of $A$ that is parallel to $v$. (So an $\bF_q$-complete matrix has at least $\frac{1}{q-1}\sum_{i=1}^{\ell}\binom{|R|}{i})$ columns.) We say a matroid $M$ is \emph{$\ell$-complete} if it is represented by an $\ell$-complete matrix $A$. Since such an $A$ has (a multiple of) each standard basis vector as a column, the rank of such an $M$ is equal to the number of rows of $A$. 

\begin{lemma}\label{stepup}
  Let $s \ge \ell \ge 3$. If $M$ is an $\ell$-complete $\bF_q$-representable matroid of rank at least $s + q^2\binom{s}{2}$, then $M$ has an $(\ell+1)$-complete minor of rank $s$. 
\end{lemma}
\begin{proof}
   Let $A$ be a $k$-complete matrix representing $M$. By possibly removing some rows and their corresponding support-one columns of $A$ to obtain a minor of $M$, we may assume that $A$ has precisely $s + q^2\binom{s}{2}$ rows. Identify the rows of $A$ with the set $R_1 \cup R_2$, where $R_1 = [s]$ and $R_2 = \bF^2 \times \binom{[s]}{2}$. Let $\{x_i : i \in [s]\} \cup \{y_{\alpha,\beta,I} : \alpha,\beta \in \bF, I \in \binom{[s]}{2}\}$ be the corresponding collection of standard basis vectors of $\bF^R$, where $R=R_1\cup R_2$. The columns of $A$ corresponding to these vectors $\{x_i\}$ and $\{y_{\alpha,\beta,I}\}$ are indexed by elements in $R$.
  
  For each $\alpha \in \bF$ and $I = \{i_1,i_2\} \in \binom{[s]}{2}$ with $i_1 < i_2$, let $v_{\alpha,\beta,I} = y_{\alpha,\beta,I} - \alpha x_{i_1} - \beta x_{i_2}$. Since $k \ge 3$ and this vector has support $3$, the matrix $A$ has a nonzero column parallel to $v_{\alpha,\beta,I}$; let $e_{\alpha,\beta,I}$ be the corresponding element of $M$; thus, $e_{\alpha,\beta,I}=\mu v_{\alpha,\beta,I}$ for some $\mu\in \bF^*$. We will contract the set $C = \{e_{\alpha,\beta,I} : \alpha,\beta \in \bF, I \in \binom{[s]}{2}\}$ to obtain the desired $(k+1)$-complete minor. 
  
  Let $A'$ be the matrix obtained from $A$ by the following sequence of operations. For each $\alpha,\beta \in \bF$ and $I=\{i,j\} \in \binom{[s]}{2}$ with $i < j$, adding $\alpha$ times the $({\alpha,\beta,I})$-row of $A$ to the $i$-row and $\beta$ times the $({\alpha,\beta,I})$-row to the $j$-row. Now, for each $I$, the $I$-column of $A'$ is a multiple of the standard basis vector $y_I$. Therefore $M/C = M[B]$, where $B$ is obtained from $A'$ by removing all the rows indexed by $I$, and removing all columns indexed by $C$.
    
  We need to show that for each $u \in \bF^s\setminus\{0\} $ of support at most $k+1$, there is a nonzero column of $B$ parallel to $u$. That is, for each such $u$, there are distinct $\ell_1, \dotsc, \ell_{k+1} \in [s]$ and some $(\lambda_1, \dotsc, \lambda_{k+1}) \in \bF^{k+1}\setminus\{0\}$ such that $u$ is nonzero outside the rows in $\{\ell_1, \dotsc, \ell_{k+1}\}$, and has entry $\lambda_i$ in row $\ell_i$. Consider the following vector in $\bF^R$: 
    \[v = y_{\lambda_1,\lambda_2,\{\ell_1,\ell_2\}} + \sum_{i=3}^{k+1} \lambda_i x_{\ell_i}. \]
  It is clear that $v$ is nonzero and has support at most $k$, and so $\mu v$ is a column of $A$ for some $\mu \in \bF^*$. Since the only nonzero entry of $v$ indexed by $R_2$ is the $y_{\ell_1,\ell_2}$ entry, it is routine to verify that the corresponding column of $A'$ is obtained from $v$ by adding $\lambda_1$ to the $\ell_1$-entry, and $\lambda_2$ to the $\ell_2$-entry, which gives the vector $v' = \mu(y_{\lambda_1,\lambda_2,\{\ell_1,\ell_2\}}  + \sum_{i=1}^{k+1} \lambda_i x_{\ell_i})$. Since $(\lambda_1, \dotsc, \lambda_{k+1}) \ne 0$, and
  all columns in $C$ are zero on the $R_1$-indexed entries, it follows that $v'$ is not parallel to any column in $C$. Removing the $R_2$ indexed entries from $v'$ gives (a multiple of) the desired vector $u$ as a column of $B$.  \qed
\end{proof}

\begin{lemma}~\lab{lem:qminor}
  Let $t \ge 3$, and $q$ be a prime power. If $A$ is a matrix over $\bF_q$ with at least $(qt)^{2^t}$ rows, and every standard basis vector and every support-three vector is a column of $A$, then $M(A)$ has a $\PG(t-1,q)$-minor. 
\end{lemma}
\begin{proof}
  For each $k \in \{3, \dotsc, t\}$, let $n_k = (q^2t)^{2^{t-k}}/q^2$. By the hypothesis, the matrix $A$ has more than $n_3$ rows; let $A'$ be obtained from $A$ by removing $\rk{A}-n_3$ rows and the corresponding standard basis vector column. 
  Now $M(A')$ is a minor of $M(A)$, and all support-three vectors are columns of $A'$, so $M(A')$ is $3$-complete with rank equal $n_3$. Let $k \in \{3, \dotsc, t\}$ be the maximum integer so that $M$ has a $k$-complete minor $M_k$ of rank at least $n_k$. 
  
  If $k < t$, then $M_k$ is $k$-complete with rank $n_k$. Now 
  \[ n_k = (q^2t)^{2^{t-k}}/q^2 = ((q^2t)^{2^{t-(k+1)}})^2/q^2 = (q^2 n_{k+1})^2/q^2 \ge n_{k+1} + q^2 \tbinom{n_{k+1}}{2}\]
  so by Lemma~\ref{stepup}, the matroid $M_k$ has a $(k+1)$-complete minor of rank $n_{k+1}$, which contradicts the maximality of $k$. 

  If $k = t$, then $M$ has a $t$-complete minor $M_t$ of rank $n_t = t$, with a $t$-complete representation $A_t$. Since every nonzero vector in $\bF_q^t$ is parallel to a column of $A_t$, we have $M_t = M(A_t) \cong \PG(t-1,q)$, so $M$ has the required minor.   \qed
\end{proof}
\medskip

 {\em Proof of Lemma~\ref{lem:minor} (when $\bF=\bF_q$)}.  Let $r=(qt)^{2^t}$.  Recall that $M'=M[A^6]$. By Lemma~\ref{lem:qminor}, it suffices to prove that a.a.s.\ $A^6$ is row equivalent to a matrix $C$ such that  the column vectors in $C$  contain every standard basis vector, and every support-three vector in $\bF_q^r$. By Lemma~\ref{lem:denseBasis}, let $u_1,\ldots, u_r$ be a $\delta$-dense basis of $B_{[r]}$ for $\bF_q^r$. Let 
 \[
 S=\{[u_1,\ldots,u_r]\alpha: \alpha \in \bF_q^r, 1\le |\supp{\alpha}|\le k  \}.
 \]
Again by the standard Chernoff bound and an analogous argument as in the proof of this lemma for the prime field case, it is easy to see that a.a.s.\ $A^6$ contains every vector in $S$. In particular, $A^6$ contains all vectors $u_1,\ldots, u_r$. Applying row operations to reduce $[u_1,\ldots, u_r]$ to the identity matrix, these row operations reduce $A^6$ to a matrix $C$ which contains all vectors in
 \[
 S'=\{\alpha \in \bF_q^r: 1\le |\supp{\alpha}|\le k \}.  
 \]
Since $k\ge 3$, $C$ contains every standard basis vector, and every support-three vector in $\bF_q^r$, as desired. \qed

\section{Proof of the key lemma: Lemma~\ref{lem:key}}\lab{sec:key}

Recall that $B=(A'')^{-1}$. As before, let $U$ denote the set of columns in $A''$ and let $U_0\subseteq U$ denote the set of columns in $\core{(A^1)}$. Let $b_1,\ldots, b_{n'}$ be the row vectors of $B$, where $n'=|U|$. Recall also that $n_1=|U_0|$ which is the number of columns in $\core{(A^1)}$, and that $n'=(1+O(\eps_1))n_1=\Theta(n)$.

For each $1\le i\le n'$, let $b_i=(b_{ij})_{1\le j\le n'}$ denote the $i$-th row vector. Let $a_1,\ldots, a_{n'}$ denote the rows of $A''$. Since $BA''=I$, we know that $b_i A'' = e_i$, where $e_i$ is the $i$-th standard basis row vector. In other words,
\[
\sum_{j=1}^{n'} b_{ij} a_j = e_i,\quad \mbox{for every $1\le i\le n'$}.
\]
It is more convenient to index the right hand side above by the columns of $A''$, and index entries of $B$ by rows and columns of $A''$ accordingly. That is, we may equivalently write
\begin{equation}\lab{row-b-index}
\sum_{j\in \Row(A'')} b_{ij} a_j= e_i,\quad \mbox{for every $i\in U$}.
\end{equation}
It is helpful to recall that $A''$ contains all but $O(\eps_1n)$ rows of $\Row(\core{(A^1)})$ by Lemma~\ref{lem:rows}.

For each row vector $a_i$, the components of $a_i$ are indexed by columns in $U$. As $U_0\subseteq U$, let $\bar a_i$ be the row vector obtained from $a_i$ by dropping the components that are not in $U_0$, for every $i\in \Row(\core{(A^1)})\supseteq \Row(A'')$. It is much easier to analyse row vectors $\bar a_i$ than $a_i$ since we know little about the distribution of the column vectors in $U\setminus U_0$. In particular, they are not independent copies of $\bfv$. By considering  $b_i \left[\substack{ \vdots \\ \bar a_j \\ \vdots}\right]$ for $i\in U_0$, it follows now that
\[
\sum_{j\in\Row(A'')} b_{ij} \bar a_j = e_i,\quad \mbox{for every $i\in U_0$}.
\]
Let $J\subseteq [r]$ and $\alpha\in \bF^J$. It follows then that
\[
 \sum_{j\in\Row(A'')}  \Big(\sum_{i\in J}  \alpha_i b_{i} \Big)_j \bar a_j = \sum_{i\in J} \alpha_i e_i.
\]
The right hand side above is a row vector that is everywhere zero except for entries in $J$. That is, 
\begin{equation}\lab{sum}
\sum_{j\in\Row(A'')}  \Big(\sum_{i\in J}  \alpha_i b_{i} \Big)_j \bar a_{jh} = 0,\quad \mbox{for all $h\in U_0\setminus J$}.
\end{equation}
Let $U_0=\{v_1,\ldots, v_{n_1}\}$ denote the set of columns in $U_0$. Let $\supp{v_h}|_{A''}$ be the set of nonzero entries in column $v_h$ of $A''$ (i.e.\ the set of $j\in\Row(A'')$ such that $\bar a_{jh}\neq 0$). Then, to satisfy~\eqn{sum}, it is necessary that
for every $h\in U_0\setminus J$, the set
\begin{equation}\lab{ZJ}
\Z^{\alpha,J}_h:=\left\{j\in \supp{v_h}|_{A''}, \Big(\sum_{i\in J}  \alpha_i b_{i} \Big)_j\neq 0\right\} 
\end{equation}
is either empty or has cardinality at least two.
 
\begin{definition}
 Let $H=H[\core{(A^1)}]$ be the hypergraph where $V(H)=\{v_1,\ldots,v_{n_1}\}$ is the set of columns of $\core{(A^1)}$, and $E(H)$ is the set of subsets $\{j\in V(H): \bar a_{ij}\neq 0\}$, for each $i\in \Row(\core{(A^1)})$. The size of the edge $\{j: \bar a_{ij}\neq 0\}$ is the cardinality of this set. An edge is called a $j$-edge if its size is $j$.
\end{definition}

Recall from~\eqn{row-b-index} that the rows of $B$ are indexed by columns in $U\supseteq U_0$.
Given any $i\in U_0$, let $H^{\alpha}_i$ denote the subgraph of $H$ obtained by including edges 
\[
\{\bar a_j: \ \alpha_ib_{ij}\neq 0\},
\]
and then deleting all the isolated vertices.
Analogously, given a subset $J\subseteq U_0$, let $H^{\alpha}_{J}$ be the subgraph obtained by including all edges $\bar a_j$ where $\sum_{i\in J} \alpha_ib_{ij}\neq 0$, and then deleting all the isolated vertices. Obviously if $\Z^{\alpha,J}_h$ has cardinality zero or at least two for every $h\in U_0\setminus J$ then all vertices in $V(H^{\alpha}_J)\setminus J$ must be incident to at least two edges in $H^{\alpha}_J$ (as there are no isolated vertices in $H^{\alpha}_J$).

Consequently, to confirm Lemma~\ref{lem:key}, it is then sufficient to prove that a.a.s.\
\begin{equation}
\mbox{for all nonzero $\alpha\in \bF^r$:}\ H^{\alpha}_{J}\ \mbox{has at least $\delta n$ edges for every $J\subseteq [r]$} \label{HJ}
\end{equation}

We prove~\eqn{HJ} by studying subgraphs of $H$ where all but at most $|J|\le r$ vertices have degree at least two.

Given a positive integer $h$, let ${\mathscr Y}^{\delta n}_h$ denote the set of subgraphs $S$ of $H$ such that 
\begin{itemize}
\item $|E(S)|\le \delta n$; and
\item exactly $h$ vertices of $S$ have degree one; and 
\item all the other vertices in $S$ have degree at least two. 
\end{itemize}
We can confirm~\eqn{HJ} if $\cup_{h=1}^{r}{\mathscr Y}^{\delta n}_{h}$ is a.a.s.\ empty. However, this is far from being true. Indeed, for every fixed $h\ge 2$ we have $|{\mathscr Y}^{\delta n}_h|=\Omega(n)$. To  see this, consider $h=2$. Clearly, any subgraph $S$ composed of a single 2-edge is a member of ${\mathscr Y}^{\delta n}_2$ and it is easy to prove that a.a.s.\ there are $\Omega(n)$ $2$-edges in $H$. Analogously, for any $h\ge 3$, we immediately find that $|{\mathscr Y}^{\delta n}_h|=\Omega(n^{\floor{h/2}})$ by only considering members of ${\mathscr Y}^{\delta n}_h$ that are some sort of pseudo-forest (see Definition~\ref{def:forest}). However, since randomly permuting vertices in $H$ does not change its distribution, we may regard $[r]$ as a random subset of  $r$ vertices of $H$. Provided that 
\begin{equation}
|{\mathscr Y}^{\delta n}_{h}|=o(n^{h}), \quad \mbox{for all $1\le h\le r$,} \lab{sparse}
\end{equation} 
the probability that the set of leaf-vertices (vertices with degree one) in any member of ${\mathscr Y}^{\delta n}_{h}$ is a subset of $[r]$ is $o(1)$. Thus it suffices to confirm~\eqn{sparse}, as~\eqn{HJ} follows by taking union bound over $1\le h\le r$. 

It is sometimes convenient to represent $H$ by a bipartite graph that is called the Tanner graph of $H$, defined below.
\begin{definition}
Given a hypergraph $G$, the Tanner graph of $G$, denoted by ${\mathcal T}_G$, is the bipartite graph on $V(G)\cup E(G)$, where $ux\in V(G)\times E(G)$ is an edge in  ${\mathcal T}_G$ if $u\in x$ in $G$. We call elements in $V(G)$ the vertex-nodes of ${\mathcal T}_G$, and elements in $E(G)$ the edge-nodes of ${\mathcal T}_G$.
\end{definition}

We start from the distribution of $H$ and the size of the edges in $H$. We first define a few parameters and a technical lemma. Let 
\[
f(x)=\frac{x(e^x-1)}{e^x-1-x}.
\]
\begin{lemma}\lab{lem:f}
$f(x)$ is an increasing function on $x>0$, and $\lim_{x\to 0+} f(x)=2$.
\end{lemma}

\proof By considering $f'(x)$, it is sufficient to prove that $e^{2x}-(2+x^2)e^x+1>0$ for every $x>0$, which follows by
\begin{align*}
e^{2x}-(2+x^2)e^x+1&=(e^x-1-x-x^2/2)^2+1+2xe^x-(1+x+x^2/2)^2\\
&> (e^x-1-x-x^2/2)^2+1+2x(1+x+x^2/2+x^3/6)-(1+x+x^2/2)^2\\
&=(e^x-1-x-x^2/2)^2+x^4/12>0.
\end{align*}
The assertion that $\lim_{x\to 0+} f(x)=2$ follows by taking the Tayler expansion of $e^x=1+x+x^2/2+O(x^3)$.\qed \ss

Given $c>2$,  let $\mu(c)$ be the unique positive root of
\[
\frac{\sum_{j\ge 2} j e^{-\mu} \mu^j/j! }{\sum_{j\ge 2}  e^{-\mu} \mu^j/j! }=f(\mu)=c.
\]
Note that the existence and uniqueness of $\mu(c)$ is guaranteed by Lemma~\ref{lem:f}.
Further, define $\beta=\beta_k$ where
\[
\beta_k=\dfrac{e^{-\mu}\mu^2/2}{\sum_{j\ge 2}  e^{-\mu} \mu^j/j!},  \quad \mbox{where $\mu=\mu(k)$.}
\]
Recall that $n_1=|V(H)|=\Theta(n)$.

\begin{lemma}\lab{lem:Hproperties}
\begin{enumerate}[(a)]
\item The number of 2-edges in $H$ is a.a.s.\ $(\beta_k+O(\eps_1))n_1$ where $\beta_k<0.45$ for every $k\ge 3$. 
\item Conditional on the number of edge-nodes in ${\mathcal T}_H$, and the degrees of all nodes in ${\mathcal T}_H$, ${\mathcal T}_H$ is uniformly distributed over all bipartite graphs with the given number of nodes and degrees of the nodes.
\item The maximum degree of ${\mathcal T}_H$ is $O(\log n)$.
\end{enumerate}
\end{lemma}

\proof By reversing the vertices and edges in $H$, the 2-edges in $H$ correspond to the vertices of degree two in the 2-core of a random $k$-uniform hypergraph where the ratio of number of vertices and the number of $k$-edges in the 2-core is $1+O(\eps_1)$. The 2-core of such a random hypergraph is well studied, and it is known that the distribution of the proportion of vertices with degree two is asymptotically truncated Poisson with parameter $\mu=(1+O(\eps_1)) \mu(k)$ where $\mu(k)$ is the root of $f(\mu)=k$. We refer the reader to~\cite[Theorem 3.3]{cain2006encores} for a proof of this assertion. Since $\mu_3 \approx 2.1491258$. It follows from Lemma~\ref{lem:f} that $\mu_k>2$ for every $k\ge 3$. The derivative of $x^2/2(e^x-1-x)$ is negative for all $x>2$. Hence, $\beta_k\le \beta_3\approx .4254370997$. This confirms part (a).
Part (b) is obvious: consider any two bipartite graphs $G_1$ and $G_2$ in the support of the probability space in part (b). Suppose $A$ is a matrix in the support of the probability space of $A^1$ such that the Tanner graph of $H[\core{A}]$ is $G_1$, then replacing $G_1$ by $G_2$ yields another matrix $B$ such that the Tanner graph of $H[\core{B}]$ is $G_2$. Since $A$ and $B$ appear with equal probability as being $A^1$, $G_1$ and $G_2$ must appear with equal probability.    Part (c) follows from well known results in the literature on the degrees of vertices of a random $k$-uniform hypergraph, and we skip its proof. \qed
\smallskip

We use the configuration model, introduced by Bollob\'{a}s~\cite{bollobas1980probabilistic}, to analyse $H$ conditioning the set of vertex-nodes, the edge-nodes, and the degree sequence. Represent each node $u$ by a bin containing $d_u$ points where $d_u$ is the degree of $u$.  Uniformly at random match points in the bins representing edge-nodes to points in the bins representing vertex-nodes. By contracting each bin to a node, and each matched pair of points as an edge, the matching yields a bipartite multigraph with  the given degree sequence. Conditional on the resulting bipartite graph being simple, the configuration model generates a bipartite graph with the same distribution as $H$. A simple counting argument as in~\cite{bollobas1980probabilistic} shows that the probability that the configuration model generates a simple bipartite graph in our setting is $\Omega(1)$. Thus, any property that is a.a.s.\ true in the configuration model is a.a.s.\ true for $H$. For simplicity, when we use the configuration model, we still call bins the vertex-nodes and edge-nodes.

We pay particular attention to the 2-edges in $H$.
For convenience, colour all  2-edges of $H$ red and let $\red{H}$ denote the subgraph of $H$ induced by the red edges. Note that~\eqn{sparse} cannot be true if $\red{H}$ has a linear component with $o(n)$ diameter. Suppose it does. Then the largest component of $\red{H}$ has a spanning tree $T$ consisting of $\Omega(n)$ vertices. For every $u,v$ in $T$, the $(u,v)$-path in $T$ is a subgraph of $H$, which is also a member of ${\mathscr Y}^{\delta n}_{2}$. It would immediately imply that $|{\mathscr Y}^{\delta n}_{2}|=\Omega(n^2)$, contradicting~\eqn{sparse}. 
Our next target is to prove that $\red{H}$ does not contain any large component. 
\begin{lemma}\label{lem:smallRedGraph} 
A.a.s\ every component of $\T_{\red{H}}$ has order $O(\log n)$. 
\end{lemma}
\proof $\T_{\red{H}}$ has $n_1$ vertex-nodes and $(\beta_k+O(\eps_1))n_1$ edge-nodes by Lemma~\ref{lem:Hproperties}. Consider the configuration model: every point in the set of degree-two edge-nodes are uniformly matched to one of the  $kn_1$ points in the vertex-nodes. Colour the points that are contained in degree-two edge-nodes red. Starting from any red point $p$, let $v$ be the vertex that contains the point that $p$ is matched to. For each of the remaining $k-1$ points $p'$ in $v$ (other than the one matched to $p$), the probability that it is matched to a red point is 
\[
\rho\sim \frac{2\beta_k}{k}+O(\eps_1).
\]
Hence, the expected number of red points by this 1-step branching process starting from $p$ is $(k-1)\rho<2\beta_k<1$ by Lemma~\ref{lem:Hproperties}(a). By a standard coupling argument comparing with a full branching process whose expected number of children in each step is less than one, the claim of the lemma follows immediately.  We leave the details of the proof as a simple exercise. \qed

\begin{definition}\lab{def:forest}
A hypergraph $G$ is called a pseudo-forest if all vertices have degree at least one, and ${\mathcal T}_G$ is a forest (i.e.\ ${\mathcal T}_G$ is acyclic). A vertex in a pseudo-forest is called a leaf if its degree is equal to one.
\end{definition}

We are ready to count ${\mathscr Y}^{\delta n}_h$. It is easy to bound the number of members in ${\mathscr Y}^{\delta n}_h$ that are pseudo-forests. Note that every pseudo-forest must have at least two leaves.
\begin{lemma}\label{lem:forests}
Let $h\ge 2$. A.a.s.\ the number of subgraphs of $H$ that are  pseudo-forests with at most $h$ leaves is $O((n\log^{2h} n)^{h/2})$. 
\end{lemma}

\proof  Let $F$ be a subgraph of $H$ which is a pseudo-forest. Then every component of $F$ must have at least two leaves. Thus, $F$ has at most $h/2$ components. We call a pseudo-forest a pseudo-tree if its tanner graph is connected. It is thus sufficient to prove that there are a.a.s.\  $O(n\log^{2h} n)$ pseudo-trees in $H$ with at most $h$ leaves.

Suppose $T$ is a pseudo-tree, we obtain $P(T)$ by contracting every 2-edges in $T$; 
we call $P(T)$ the profile of $T$. See Figure~\ref{fig:tree} below for an illustration. The round nodes are vertex-nodes and the square nodes are edge-nodes. Edge-nodes with degree two are coloured red. After contracting all the 2-edges, i.e.\ contracting all red edge-nodes in the Tanner graph, vertices $v_1$, $v_2$ and $v_3$ are merged to a single vertex, labelled as $v_1$ in the figure on the right hand side. After contraction, all red square nodes disappear. Notice that $P(T)$ has the same number of leaves as $T$. Suppose that $T$ is a pseudo-tree with $h$ leaves. It is easy to see that $P(T)$ has at most $h$ edges, since every edge in $P(T)$ has size at least three, and $P(T)$ has $h$ leaves.  It follows now that for a fixed  $h$, there are at most $O(1)$ possible profiles (i.e.\ pseudo-trees where every edge has size at least three) with $h$ leaves. Fix a profile $P$, we prove that the number of pseudo-trees $T$ in $H$ such that $P(T)=P$ is $O(n\log^ {2h} n)$.  Then summing over all $O(1)$ possible $P$ gives our desired bound on the number of pseudo-trees in $H$.

\begin{figure}
\begin{center}
 \includegraphics[scale=1]{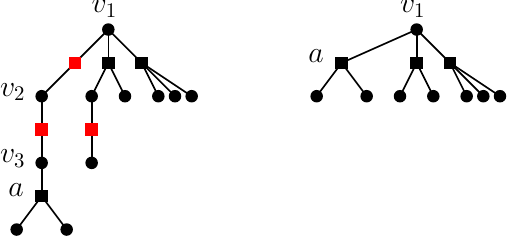}
 \caption{The figure on the left is the Tanner graph of a pseudo-tree $T$; the figure on the right is the Tanner graph of $P(T)$.}
  \label{fig:tree}
 \end{center}
\end{figure}

Pick an arbitrary vertex $u$ in $P$. Obviously there are at most $n_1\le n$ ways to choose $u$  in  $H$. Enumerate all edges $e_1,\ldots, e_{\ell}$, ($\ell\le h$)  in $P$   by a BFS-tree starting from $u$. We embed these edges or their counterparts in $H$ (corresponding to certain paths in $\T_H$ defined below)  one at a time, and count the number of choices for $e_i$, given the embedding of $e_1,\ldots, e_{i-1}$. Suppose the parent of the edge-node $e_i$ in $\T_P$ is vertex-node $v$, and suppose that $e_1,\ldots, e_{i-1}$ has already been embedded into $H$. That means, $v$ is already embedded into $H$. To embed $e_i$ into $H$, we need to find an edge $e$ in $H$ whose size is equal to $e_i$, and a $(v,e)$-path in $\T_H$ such that all the internal edge-nodes are red. For example, take $P$ whose Tanner graph $\T_P$ is given by the right hand side of Figure~\ref{fig:tree}. Given $v_1$ embedded into $H$ already, in order to embed the edge $a$, we need to find a $(v_1,a)$-path in $\T_H$ as shown on the left hand side of Figure~\ref{fig:tree}. We prove that for each $i$, there are $O(\log^2 n)$ ways to embed $e_i$.

\begin{claim}\lab{c:profile}
There are $O(\log^2 n)$ choices for $(e,P)$ in $H$, such that $e$ has the same size as $e_i$, and $P$ is a $(v,e)$-path in $\T_H$ such that all the internal edge-nodes are red. 
\end{claim}

Since there are $h$ edges in $P(T)$, the lemma follows immediately. It remains to prove the claim. 

{\em Proof of Claim~\ref{c:profile}. } By Lemma~\ref{lem:smallRedGraph}, there are $O(\log n)$ choices for a $(v,u)$-path in $\T_H$ such that $u$ is a vertex node, and all the internal edge-nodes are red. Then, by Lemma~\ref{lem:Hproperties}(c), given $u$, there are $O(\log n)$ choices for an edge $e$ incident to $u$ such that the size of $e$ is equal to the size of $e_i$. Combining them together yields the claim. \qed 
\ss

Next we treat members of ${\mathscr Y}^{\delta n}_h$ that are not pseudo-forests. Suppose $S\in {\mathscr Y}^{\delta n}_h$ and ${\mathcal T}_S$ contains a cycle. Then, we can repeatedly remove leaves from  $\T_S$ until we obtain a subgraph of $\T_S$ whose minimum degree is at least two. This subgraph is known as the 2-core of $\T_S$, denoted by $\core{(\T_S)}$. The notation of the 2-core of a graph and hypergraph is analogous to Definition~\ref{def:2core} by treating $A$ as the transpose of a weighted incidence matrix of a hypergraph. Since ${\mathcal T}_S$ contains a cycle, $\core{({\mathcal T}_S)}$ is nonempty. However, the hypergraph with Tanner graph $\core{({\mathcal T}_S)}$ is not necessarily a subgraph of $H$, and in particular, $\core{(\T_S)}$ is not necessarily $\T_{\core{H}}$. This is because that it is possible that there is some edge-node $x$ in $\core{({\mathcal T}_S)}$ with smaller degree than its degree in ${\mathcal T}_H$, and thus the edge corresponding to $x$ in the hypergraph corresponding to the Tanner graph $\core{({\mathcal T}_S)}$ is not an edge in $H$. Look at the example in Figure~\ref{fig:2core}.
The bipartite graph $\T_G$ in the middle is the Tanner graph of the hypergraph $G$ on the left, and the bipartite graph  on the right is $\core{(\T_G)}$, the 2-core of $\T_G$. The edge-node $a$ in $\T_G$ has degree 3, whereas it has degree 2 in $\core{(\T_G)}$ after the removal of $u$, and thus the edge-node $a$ in $\core{(\T_G)}$ does not correspond to an edge of $G$ any more. Indeed, the 2-core of $G$ is empty, whereas $\core{(\T_G)}$ is nonempty. Obviously, 
 to relate $\core{({\mathcal T}_S)}$ to the ``correct subgraph'' of $S$ (which is not $\T_{\core{S}}$ as illustrated by the above example), we need to treat edge-nodes in $\core{({\mathcal T}_S)}$ whose degrees are smaller than their degrees in $S$.

\begin{figure}
\begin{center}
\includegraphics[scale=0.6]{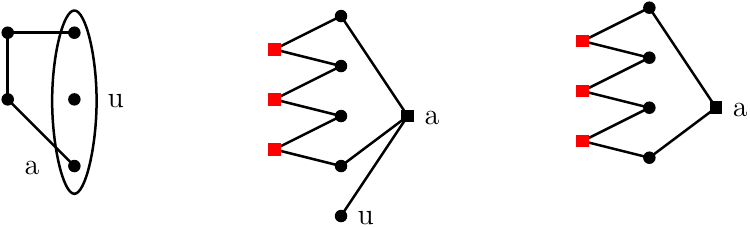}
\caption{An example of a hypergraph, its Tanner graph, and the 2-core of its Tanner graph}
\label{fig:2core}
\end{center}
\end{figure}

\begin{definition}
Given a subgraph $S'$ of a Tanner graph $S$, we say an edge-node $x$ in $S'$ has excess $t$, if $\deg_{S}(x)=\deg_{S'}(x)+t$. The excess of $S'$ is defined by the total excess of all edge-nodes in $S'$. The closure of $S'$, denoted by $\cl{S'}$, is the minimum Tanner graph such that $S'\subseteq \cl{S'}\subseteq S$, $S'$ and $\cl{S}$ have the same set of edge-nodes, and $\deg_{S}(x)=\deg_{\cl{S'}}(x)$ for every edge-node $x$ of $\cl{S'}$.
\end{definition}

For example, let $S$ be the Tanner graph in the middle of Figure~\ref{fig:2core}, and let $S'$ be the graph on the right hand side of Figure~\ref{fig:2core}. Then, edge-node $a'$ in $S'$ has excess 1, and all the other edge-nodes in $S'$ has excess 0. In this example, $\cl{S'}=S$.

We will count non-pseudo-forest members of ${\mathscr Y}^{\delta n}_h$ in an inside-out manner. First we bound the number of subgraphs $S'$ of $\T_H$ that can potentially be $\core{(\T_S)}$ for some $S\in {\mathscr Y}^{\delta n}_h$. Then, we start from  such an $S'$, and find all possible subgraphs $G$ of $H$ such that $G\in {\mathscr Y}^{\delta n}_h$ and $\core{(\T_G)}=S'$. Recall the procedure of obtaining the 2-core of a Tanner graph by repeatedly removing leaves. Recovering $G$ from $\core{(\T_G)}=S'$ can be done by first finding $\cl{S'}$ and then apply a reverse procedure, which repeatedly extends $\cl{S'}$ by attaching some tree-structures. We formally define the procedure below.

Given a subgraph $S$ of a hypergraph $H$, we say that we {\em attach a leaf-edge $x$ to $S$}, if we obtain a subgraph $S\subseteq S'\subseteq H$ by adding into $S$ a set of vertices $v_1,\ldots,v_t\notin V(S)$ and an edge $x\in E(H)$, where $x=\{v_0,v_1,\ldots,v_t\}$ for some $v_0\in V(S)$. A subgraph $S'$ of $H$ is called a {\em tree-extension} of $S$, if $S'$ is obtained from $S$ by repeatedly attaching leaf-edges. Part (a) of the following observation is obvious by the leaf-removal process to obtain the 2-core of $\T_G$ for a hypergraph $G$. Part (b) follows by the fact that the number of leaves does not decrease in the process of tree-extensions.

\begin{obs} Let $h\ge 1$ be fixed.
\begin{enumerate}[(a)]
\item Suppose $S\in {\mathscr Y}^{\delta n}_h$ and $\mathcal {T}_S$ contains a cycle. Then $S$ is a tree-extension of $S'$, where $S'$ is the hypergraph corresponding to $\cl{\core{(\mathcal {T}_S)}}$.
\item Suppose $S\in {\mathscr Y}^{\delta n}_h$ and $\mathcal {T}_S$ contains a cycle, then $\core{({\mathcal {T}_S})}$ has excess at most $h$. 
\end{enumerate}
\end{obs}


If  $S\in {\mathscr Y}^{\delta n}_h$ then  $\core{({\mathcal T}_S)}$ is a subgraph of ${\mathcal T}_H$ in which every node (i.e.\ every vertex-node and every edge-node) has degree at least two. This motivates the consideration of the following set of subgraphs of $\T_H$.
Let $h\ge 0$ be fixed. Let ${\mathcal A}^h(x,y_2,y_{\ge 3},t)$ denote the set of subgraphs $S$ of ${\mathcal T}_H$ such that
\begin{itemize}
\item $S$ has exactly $x$ vertex-nodes , each of which has degree at least two in $S$;
\item $S$ has exactly $y_2$ edge-nodes whose degree in $\mathcal{T}_H$ is 2, and $y_{\ge 3}$ edge-nodes whose degree in $\mathcal{T}_H$ is at least three;
\item $S$ has total excess $h$;
\item $S$ has exactly $t$ edges.
\end{itemize}
We bound the number of non-pseudo-forest members of ${\mathscr Y}^{\delta n}_h$ by bounding the cardinality of ${\mathcal A}^h(x,y_2,y_{\ge 3},t)$ and the number of tree-extensions of any member of ${\mathcal A}^h(x,y_2,y_{\ge 3},t)$.
\begin{lemma}\lab{lem:corecounts} Let $h\ge 0$ be fixed and $\delta>0$ be a sufficiently small real number.
\begin{enumerate}[(a)]
\item $\sum_{t\le \delta n_1}\sum_{x, y_2,y_{\ge 3}}\ex|{\mathcal A}^h(x,y_2,y_{\ge 3},t)| = O((\log^{2h+7} n)  n^{h/2}).$
\item There exists a sufficiently large constant $K>0$ such that $\sum \ex|{\mathcal A}^h(x,y_2,y_{\ge 3},t)| = o(1)$, where the summation is over all $(x,y_2,y_{\ge 3},t)$ such that $t\le \delta n_1$ and $\max\{y_2,y_{\ge 3}\}\ge K\log^2 n$.
\item For every fixed $K>0$, a.a.s.\ for every 
\[S\in\cup_{j=0}^h\cup_{\substack{(x,y_2,y_{\ge 3},t)\\ t\le K\log^3 n}}\mathcal {A}^j(x,y_2,y_{\ge 3},t),
\] the subgraph of $H$ with Tanner graph $\cl{S}$ has $O(\log^{2h+3} n)$ tree extensions, each of which has $h$ vertices with degree one.
\end{enumerate}
\end{lemma}

Using Lemma~\ref{lem:corecounts} we confirm~\eqn{sparse} in the following lemma.
\begin{lemma} \lab{lem:Ybound}
For all sufficiently small $\delta>0$ and every fixed integer $h\ge 1$, a.a.s.\ $|{\mathscr Y}^{\delta n}_h|=O(n^{h/2+o(1)})$.
\end{lemma}

\proof Let $K>0$ be the constant that makes the assertion in Lemma~\ref{lem:corecounts}(b) holds. Then, the probability of having $S\in {\mathscr Y}^{\delta n}_h$ such that $\core{(\T_S)}$ has more than $K\log^2 n$ edge-nodes of degree two, or more than $K\log^2 n$ edge-nodes of degree at least three, is $o(1)$ by Markov's inequality. Now consider $S\in {\mathscr Y}^{\delta n}_h$ such that  $\core{(\T_S)}$ is nonempty and  has at most $2K\log^2 n$ edge-nodes. By Lemma~\ref{lem:Hproperties}(c), we may assume that $\core{(\T_S)}$ has $O(K\log^3 n)=O(\log^3 n)$ edges. 
By Lemma~\ref{lem:corecounts}(a,c),
the expected number of such $S\in {\mathscr Y}^{\delta n}_h$ is bounded by $O((\log^{4h+10} n)n^{h/2})=O(n^{h/2+o(1)})$. The assertion now follows by combining the above two cases and Lemma~\ref{lem:forests}. \qed

\subsection{Proof of Lemma~\ref{lem:key}}
By Lemma~\ref{lem:Ybound} and the argument above~\eqn{sparse}, a.a.s.\ there exists constant $\delta>0$ such that 
$\E$ holds, where $\E$ denotes the event that there is no subgraph $S$ of $H$ such that $S$ has less than $\delta n$ edges and has exactly at most $r$ vertices of degree equal to 1, all of which are contained in $[r]$. Suppose event $\E$ holds for $H$. Suppose on the contrary to~\eqn{HJ} that there exists nonzero $\alpha\in \bF^r$ such that $H_J^{\alpha}$ has less than $\delta n$ edges. But $H_J^{\alpha}$ must be a subgraph of $H$ such that all vertices in $V(H_J^{\alpha})\setminus J$ have degree at least two, as shown above~\eqn{HJ}. In other words, $H_J^{\alpha}$ is a subgraph of $H$ with less than $\delta n$ edges, and all vertices with degree equal to 1 is contained in $[r]$. This contradicts with the event $\E$.\qed

\subsection{Proof of Lemma~\ref{lem:corecounts}.} 

We first prove part (c). Given any \[S\in\cup_{j=0}^h\cup_{\substack{(x,y_2,y_{\ge 3},t)\\ t\le K\log^3 n}}\mathcal {A}^j(x,y_2,y_{\ge 3},t),\] there are  $O(\log^3 n)$ edges in $S$, and thus  $O(\log^3 n)$ total vertex-nodes and edge-nodes in $S$ and $\cl{S}$. Immediately, the subgraph $G$ of $H$ such that  $\T_G=\cl{S}$ has $O(\log^3 n)$ vertices. Following the same proof as in Lemma~\ref{lem:forests}, 
there are $O(\log^{2h} n)$ tree extensions starting from any given vertex in $G$. Thus, the total number of tree extensions is $O(\log^{2h+3}n)$.

Next, we prove parts (a) and (b). We use the configuration model for the analysis.
Given any $S\in {\mathcal A}^h(x,y_2,y_{\ge 3})$, let $y_j$ denote the number of edge-nodes in $S$ whose degree in ${\mathcal T}_H$ is $j$, for $j\ge 3$.
 Then immediately $t=\sum_{j\ge 2}jy_j-h$.  Moreover, every vertex-nodes in $S$ has degree at least two, and at most $k$. Thus, necessarily, $2x\le t\le kx$. We choose vertex-nodes and edge-nodes for $S\in {\mathcal A}^h(x,y_2,y_{\ge 3},t)$. Let $\beta=\beta_k$. By Lemma~\ref{lem:Hproperties}(a), there are at most $(\beta+O(\eps_1))n_1$ edge-nodes in ${\mathcal T}_H$ with degree two. Thus, 
there are at most $\binom{(1+O(\eps_1))n_1}{x}$ ways to choose the vertex-nodes, at most $\binom{(\beta+O(\eps_1)) n_1}{y_2}$ ways to choose edge-nodes that have degree two in ${\mathcal T}_H$, and  at most $\binom{(1+O(\eps_1))n_1}{y_{\ge 3}}$ ways to choose the edge-nodes whose degree in ${\mathcal T}_H$ is at least three. Given the choices of all vertex-nodes and edge-nodes, there are at most $(\sum_{j\ge 2}jy_j)^h=(t+h)^h=O(t^h)$ ways to choose the $h$ points in the edge-nodes of $S$ that are not matched to any points in the vertex-nodes of $S$. Then, there are at most $\binom{k}{2}^x \binom{kx-2x}{t-2x}$ ways to choose $t$ points from the vertex-nodes of $S$ so that every vertex-nodes contain at least two of the $t$ chosen points. Finally, there are $t!$ ways to match the $t$ points in the vertex-nodes of $S$ to the $t$ points in the edge-nodes. The probability of the appearance of every such $t$ pairs of points is
\[
\prod_{i=0}^{t-1}\dfrac{1}{(kn_1-2i-1)}\le (kn_1-2t)^{-t}.
\]
Hence,
\begin{align*}
\ex |{\mathcal A}(x,y_2,y_{\ge 3},t)| \le & \binom{(1+O(\eps_1))n_1}{x} \binom{(\beta+O(\eps_1)) n_1}{y_2} \binom{(1+O(\eps_1))n_1}{y_{\ge 3}} O(t^h) \binom{k}{2}^x \binom{kx-2x}{t-2x} \\
&\times t!  (kn_1-2t)^{-t}.
\end{align*}
Let $M=(1+O(\eps_1))n_1$ and define
\[
g(x)=\binom{(1+O(\eps_1))n_1}{x}\binom{k}{2}^x \binom{kx-2x}{t-2x}.
\]
Then, for any $1\le x\le t/2-1$ where $t\le \delta n_1$,
\[
\frac{g(x+1)}{g(x)}=\frac{M-x}{x+1}\binom{k}{2}(t-2x)_2\frac{((k-2)(x+1))_{t-2x-2}}{((k-2)x)_{t-2x}}\ge \frac{1}{\delta} \left(\frac{t-2x-1}{kx-t+2}\right)^2.
\]
The above ratio is at least one if $t-2x\ge 2k\sqrt{\delta}x$. On the other hand, $x\le t/2$ is necessary for ${\mathcal A}(x,y_2,y_{\ge 3},t)$ to be nonempty.
Consequently, $g(x)$ is maximised at some 
\[
x^*=(1/2+O(\delta^{1/2}))t, \quad \text{and}\ x^*\le t/2.
\]
Thus, by choosing sufficiently small $\delta$,
\begin{align*}
\ex |{\mathcal A}(x,y_2,y_{\ge 3},t)| = & O(t^h)  \binom{(1+O(\eps_1))n_1}{t/2} \binom{(\beta+O(\eps_1)) n_1}{y_2} \binom{(1+O(\eps_1))n_1}{y_{\ge 3}}  \binom{k}{2}^{t/2} \binom{kx-2x^*}{t-2x^*} \\
&\times t!  (kn_1-2t)^{-t}\\
=&O(t^h)  \binom{(1+O(\eps_1))n_1}{t/2} \binom{(\beta+O(\eps_1)) n_1}{y_2} \binom{(1+O(\eps_1))n_1}{y_{\ge 3}} \left(\dfrac{k^2}{2}\right)^{t/2} \\
& \times \binom{kt/2}{O(\delta^{1/2} t)} t!  (kn_1(1+O(\delta)))^{-t}.
\end{align*}
Using Stirling's formula, the inequality $\binom{y}{z}\le (ey/z)^z$ and by setting $(y/0)^0=1$ for any real $y$, we obtain
\begin{align*}
\binom{(1+O(\eps_1))n_1}{t/2} \le &  \left(\dfrac{(2+O(\eps_1))en_1}{t}\right)^{t/2}\\
\binom{(\beta+O(\eps_1)) n_1}{y_2} \le &  \left(\dfrac{e(\beta +O(\eps_1))n_1}{y_2}\right)^{y_2}\\
 \binom{(1+O(\eps_1))n_1}{y_{\ge 3}} \le & \left(\dfrac{(e+O(\eps_1))n_1}{y_{\ge 3}}\right)^{y_{\ge 3}}\\
 \binom{kt/2}{O(\delta^{1/2} t)} \le & \exp(O(\delta^{1/2}\log\delta^{-1/2})t)\\
 t! = & O(t^{1/2})\left(t/e\right)^t.
\end{align*}
Hence,
\begin{align*}
\ex |{\mathcal A}(x,y_2,y_{\ge 3},t)| \le & O(t^{h+1/2}) \left( \exp(O(\delta^{1/2}\log \delta^{-1/2}+\eps_1)) \left(\frac{t}{en_1}\right)^{1/2} \right)^t \left(\dfrac{e(\beta +O(\eps_1))n_1}{y_2}\right)^{y_2} \\
&\times \left(\dfrac{(e+O(\eps_1))n_1}{y_{\ge 3}}\right)^{y_{\ge 3}}.
\end{align*}
Therefore, there exists $K=\exp(O(\delta^{1/2}\log \delta^{-1/2}+\eps_1))$ such that
\begin{equation}
\ex |{\mathcal A}(x,y_2,y_{\ge 3},t)| \le  O(t^{h+1/2}) \left( K \left(\frac{t}{en_1}\right)^{1/2} \right)^t \left(\dfrac{e(\beta +O(\eps_1))n_1}{y_2}\right)^{y_2}  \left(\dfrac{(e+O(\eps_1))n_1}{y_{\ge 3}}\right)^{y_{\ge 3}}. \lab{eq:Abound}
\end{equation}
The derivative of $\log \left(t^{h+1/2} \left( K \left(\frac{t}{en_1}\right)^{1/2} \right)^t\right)$ is
\[
\frac{h+1/2}{t} + \log K + \frac{1}{2}\log (t/en_1) +\frac{1}{2}<0,
\]
as $t<\delta n_1$ where $\delta$ is sufficiently small. So the right hand side of~\eqn{eq:Abound} is a decreasing function of $t$. Since 
\[
t=\sum_{j\ge 2}jy_j-h\ge 2y_2+3y_{\ge 3}-h,
\]
the right hand side of~\eqn{eq:Abound} is maximised at $t=2y_2+3y_{\ge 3}-h$. Hence, using $\eps_1=o(\delta)$ and thus $K=1+O(\delta^{1/4})$,
\begin{align*}
\ex |{\mathcal A}(x,y_2,y_{\ge 3},t)| = & O((y_2+y_{\ge 3})^{h+1/2} n_1^{h/2}) \left(K^2\cdot \dfrac{2y_2+3y_{\ge 3}}{en_1} \cdot \frac{e(\beta+O(\eps_1))n_1}{y_2}\right)^{y_2}\\
&\times \left(K^3\left(\dfrac{2y_2+3y_{\ge 3}}{en_1}\right)^{3/2} \cdot \frac{(e+O(\eps_1))n_1}{y_{\ge 3}}\right)^{y_{\ge 3}}\\
= & O((y_2+y_{\ge 3})^{h+1/2} n_1^{h/2})   \left(\dfrac{\beta (2y_2+3y_{\ge 3})}{y_2(1+O(\delta^{1/4}))}\right)^{y_2} \left(\dfrac{(2y_2+3y_{\ge 3})^{3/2}}{y_{\ge 3}\sqrt{en_1}(1+O(\delta^{1/4}))}\right)^{y_{\ge 3}}.
\end{align*}
Let $0<c<1$ be a constant such that $\beta(2+3c)<0.95$. We consider two cases.

{\em Case 1: $y_{\ge 3}\le cy_2$.} By choosing sufficiently small $\delta>0$,
\begin{align}
\ex |{\mathcal A}(x,y_2,y_{\ge 3},t)| = O(y_2^{h+1/2} n_1^{h/2}) \cdot  0.96^{y_2} \left(\dfrac{C_1 y_2^{3/2}}{y_{\ge 3}\sqrt{n_1}}\right)^{y_{\ge 3}}\quad \mbox{for some constant $C_1>0$.}\label{eq:case1}
\end{align}
Notice that
\begin{align*}
&\sum_{y_2\le \delta n_1} \sum_{y_{\ge 3}<cy_2} \sum_{t\le \delta n_1} \sum_{t/k\le x\le t/2}\ex |{\mathcal A}(x,y_2,y_{\ge 3},t)| = {\sum}_1+{\sum}_2,
\end{align*}
where
\begin{align*}
{\sum}_1&=\sum_{0\le y_2< \log^2 n_1} \sum_{y_{\ge 3}<cy_2} \sum_{t\le \delta n_1} \sum_{t/k\le x\le t/2} \ex |{\mathcal A}(x,y_2,y_{\ge 3},t)|\\
{\sum}_2&=\sum_{\log^2 n_1\le y_2\le \delta n_1} \sum_{y_{\ge 3}<cy_2} \sum_{t\le \delta n_1} \sum_{t/k\le x\le t/2} \ex |{\mathcal A}(x,y_2,y_{\ge 3},t)| 
\end{align*}
For the first sum above, both $y_2$ and $y_{\ge 3}$ are $O(\log^2 n)$. By Lemma~\ref{lem:Hproperties}(c), we may assume that $t=O(\log^3 n)$ as otherwise ${\mathcal A}(x,y_2,y_{\ge 3},t)$ is a.a.s.\ empty.  By~\eqn{eq:case1},
\begin{align*}
{\sum}_1 &= O(\log^6 n_1) \sum_{0\le y_2< \log^2 n_1} \sum_{y_{\ge 3}<cy_2}  \ex |{\mathcal A}(x,y_2,y_{\ge 3},t)|\nonumber\\
&=O(\log^{2h+7} n_1)n_1^{h/2}\sum_{0\le y_2< \log^2 n_1}0.96^{y_2} \sum_{0\le y_{\ge 3}<cy_2}    \left(\dfrac{C_1 \log^{3}n_1}{y_{\ge 3}\sqrt{n_1}}\right)^{y_{\ge 3}}\nonumber\\
&=O(\log^{2h+7} n_1)n_1^{h/2}.
\end{align*}
For $\sum_2$, note that $t=O(y_2\log n_1)$ and thus $x=O(y_2\log n_1)$. Thus,
\begin{align*}
{\sum}_2
=&O(n_1^{h/2})\sum_{\log^2 n_1\le y_2\le \delta n_1} \sum_{y_{\ge 3}<cy_2} y_2^2\log^2 n_1 \cdot y_2^{h+1/2} \cdot  0.96^{y_2} \left(\dfrac{C_1 y_2^{3/2}}{y_{\ge 3}\sqrt{n_1}}\right)^{y_{\ge 3}}.
\end{align*}
The above function is maximised at $y_{\ge 3}=C_1 y_2^{3/2}/e\sqrt{n_1}$. So
\begin{align*}
{\sum}_2
=&O(n_1^{h/2}) n_1\cdot \log^2 n_1 \sum_{\log^2 n_1\le y_2\le \delta n_1}  y_2^{h+5/2} \cdot  0.96^{y_2} \exp\left(C_1 y_2^{3/2}/e\sqrt{n_1}\right)\\
=&O(n_1^{h/2}) n_1\cdot \log^2 n_1 \sum_{\log^2 n_1\le y_2\le \delta n_1}  y_2^{h+5/2}  \exp\left(y_2(\ln(0.96)+(C_1/e) \sqrt{y_2/n_1})\right).
\end{align*}
Since $y_2\le \delta n_1$ and by choosing sufficiently small $\delta$,
\begin{align}
{\sum}_2
=&O(n_1^{h/2}) n_1\cdot \log^2 n_1 \sum_{\log^2 n_1\le y_2\le \delta n_1}  y_2^{h+5/2}  \exp\left(-\Theta(y_2)\right) =o(1). \lab{eq:large-y2-1}
\end{align}
It follows now that
\begin{align}
&\sum_{y_2\le \delta n_1} \sum_{y_{\ge 3}<cy_2} \sum_{t\le \delta n_1} \sum_{t/k\le x\le t/2}\ex |{\mathcal A}(x,y_2,y_{\ge 3},t)| =O(\log^{2h+7} n_1)n_1^{h/2}. \lab{eq:sum1}
\end{align}


{\em Case 2: $y_{\ge 3}> cy_2$.} Note that this immediately implies that $y_{\ge 3}\ge 1$.  In this case, for some constants $C_2,C_3>0$,
\begin{align*}
\ex |{\mathcal A}(x,y_2,y_{\ge 3},t)| &= O(y_{\ge 3}^{h+1/2} n_1^{h/2})   \left(\dfrac{C_2 y_{\ge 3}}{y_{2}}\right)^{y_2} \left(\dfrac{C_3 y_{\ge 3}^{1/2}}{\sqrt{n_1}}\right)^{y_{\ge 3}}. \label{eq:case2}
\end{align*}
Similarly as before, $ \left(\dfrac{C_2 y_{\ge 3}}{y_{2}}\right)^{y_2}$ is maximised at $y_2=C_2y_{\ge 3}/e$. Thus,
\begin{align*}
\ex |{\mathcal A}(x,y_2,y_{\ge 3},t)| &= O(y_{\ge 3}^{h+1/2} n_1^{h/2})  \exp \left(\dfrac{C_2 y_{\ge 3}}{e}\right) \left(\dfrac{C_3 y_{\ge 3}^{1/2}}{\sqrt{n_1}}\right)^{y_{\ge 3}}\\
&=O(y_{\ge 3}^{h+1/2} n_1^{h/2}) \exp \left(y_{\ge 3} \left(\dfrac{C_2 }{e}+\frac{1}{2}\log(y_{\ge 3}/n_1)+\log C_3\right)\right).
\end{align*}
Hence, $\sum_{y_{\ge 3}\ge 1} \sum_{y_{2}\le c^{-1}y_{\ge 3}} \sum_{t\le \delta n_1} \sum_{t/k\le x\le t/2}\ex |{\mathcal A}(x,y_2,y_{\ge 3},t)|=\Sigma_3+\Sigma_4$ where
\begin{eqnarray}
{\sum}_3&=&\sum_{1\le y_{\ge 3}\le \log^2 n_1} \sum_{y_{2}\le c^{-1}y_{\ge 3}} \sum_{t\le \delta n_1} \sum_{t/k\le x\le t/2}\ex |{\mathcal A}(x,y_2,y_{\ge 3},t)|  \nonumber\\
&=&O(n_1^{h/2})\sum_{1\le y_{\ge 3}\le \log^2 n_1} \sum_{y_{2}\le c^{-1}y_{\ge 3}} \sum_{t\le \delta n_1} \sum_{t/k\le x\le t/2} y_{\ge 3}^{h+1/2}   \exp \left(\dfrac{C_2 y_{\ge 3}}{e}\right) \left(\dfrac{C_3 y_{\ge 3}^{1/2}}{\sqrt{n_1}}\right)^{y_{\ge 3}} \nonumber\\
&=&O(n_1^{h/2})\log^8 n_1\sum_{1\le y_{\ge 3}\le \log^2 n_1}  y_{\ge 3}^{h+1/2}   \exp \left(\dfrac{C_2 y_{\ge 3}}{e}\right) \left(\dfrac{C_3 y_{\ge 3}^{1/2}}{\sqrt{n_1}}\right)^{y_{\ge 3}}\nonumber\\
&=&O(n_1^{(h-1)/2}\log^8 n_1)=O(n_1^{(h-1)/2}\log^8 n_1), \label{eq:sum2}
\end{eqnarray}
and
\begin{eqnarray}
{\sum}_4&=&\sum_{y_{\ge 3}\ge \log^2 n_1} \sum_{y_{2}\le c^{-1}y_{\ge 3}} \sum_{t\le \delta n_1} \sum_{t/k\le x\le t/2}\ex |{\mathcal A}(x,y_2,y_{\ge 3},t)| \nonumber\\
&=&O(n_1^{h/2})\sum_{y_{\ge 3}> \log^2 n_1} \sum_{y_{2}\le c^{-1}y_{\ge 3}} \sum_{t\le \delta n_1} \sum_{t/k\le x\le t/2} O(y_{\ge 3}^{h+1/2}) \exp \left(y_{\ge 3} \left(\dfrac{C_2 }{e}+\frac{1}{2}\log(y_{\ge 3}/n_1)+\log C_3\right)\right)\nonumber\\
&=&O(n_1^{h/2}) n_1^3\sum_{y_{\ge 3}> \log^2 n_1} O(y_{\ge 3}^{h+1/2}) \exp \left(-\Omega(y_{\ge 3})\right)=o(1). \label{eq:large-y3-1} 
\end{eqnarray}
Now part (a) of the lemma follows by~\eqn{eq:sum1} and~\eqn{eq:sum2}, and part (b) of the lemma follows by~\eqn{eq:large-y2-1} and~\eqn{eq:large-y3-1} and by choosing $K=1/c$.
 \qed


 \end{document}